\documentclass[11pt]{amsart}
\pdfoutput=1

\usepackage{amssymb}
\usepackage{microtype}
\usepackage{stmaryrd}
\usepackage{mathtools}
\usepackage{tikz-cd}
\usepackage{enumitem}

\usepackage[hmargin=1.25in,vmargin=1in,marginparwidth=1in]{geometry}

\usepackage[numbers]{natbib}

\usepackage{hyperref}
\hypersetup{
  pdftitle={Functional equivariance and conservation laws in numerical integration},
  pdfauthor={Robert I. McLachlan and Ari Stern},
  bookmarksopen=false
}

\usepackage[capitalize,nameinlink]{cleveref}

\theoremstyle{plain}
\newtheorem{theorem}{Theorem}[section]
\newtheorem{proposition}[theorem]{Proposition}
\newtheorem{corollary}[theorem]{Corollary}

\theoremstyle{definition}
\newtheorem{definition}[theorem]{Definition}
\newtheorem{example}[theorem]{Example}
\newtheorem{assumption}[theorem]{Assumption}

\theoremstyle{remark}
\newtheorem{remark}[theorem]{Remark}

\begin{document}

\title{Functional equivariance and conservation laws in numerical
  integration}

\author{Robert I.~McLachlan}
\address{School of Fundamental Sciences, Massey University}
\email{r.mclachlan@massey.ac.nz}

\author{Ari Stern}
\address{Department of Mathematics and Statistics,
  Washington University in St.~Louis}
\email{stern@wustl.edu}

\begin{abstract}
  Preservation of linear and quadratic invariants by numerical
  integrators has been well studied. However, many systems have linear
  or quadratic observables that are not invariant, but which satisfy
  evolution equations expressing important properties of the
  system. For example, a time-evolution PDE may have an observable
  that satisfies a local conservation law, such as the multisymplectic
  conservation law for Hamiltonian PDEs.

  We introduce the concept of \emph{functional equivariance}, a
  natural sense in which a numerical integrator may preserve the
  dynamics satisfied by certain classes of observables, whether or not
  they are invariant. After developing the general framework, we use
  it to obtain results on methods preserving local conservation laws
  in PDEs. In particular, integrators preserving quadratic invariants
  also preserve local conservation laws for quadratic observables, and
  symplectic integrators are multisymplectic.
\end{abstract}

\maketitle

\section{Introduction}

In numerical ordinary differential equations (ODEs), it is known that
all B-series methods (including Runge--Kutta methods) preserve linear
invariants, while only certain ones preserve quadratic
invariants. Linear invariants arising in physical systems include
mass, charge, and linear momentum; quadratic invariants include
angular momentum and other momentum maps, as well as the canonical
symplectic form for Hamiltonian systems. See \citet*{HaLuWa2006} and
references therein.

However, for partial differential equations (PDEs) describing time
evolution, it is desirable for a numerical integrator to preserve not
only global invariants but also local conservation laws. For instance,
the evolution may preserve total mass (a global invariant), but the
mass in a particular region may change by flowing through the boundary
of the region (a local conservation law). Another example is the
canonical multisymplectic conservation law for Hamiltonian PDEs, which
is a quadratic local conservation law for the variational
equation. Focusing only on global invariants overlooks this more
granular, local form of conservativity.

This paper develops a new framework for the preservation of such
properties by numerical integrators. We do so by answering a much more
general question: When does a numerical integrator preserve the
evolution of certain classes of observables (e.g., linear, quadratic),
even when those observables are not invariants?  This includes not
only global invariants, as previously studied, but also local
conservation laws and other balance laws encountered in both
conservative and dissipative dynamical systems.

The main idea of our approach is summarized as follows. Suppose
$ y = y (t) $ evolves in a (finite- or infinite-dimensional) Banach
space $Y$ according to $ \dot{y} = f (y) $. Given a functional
$ F \in C ^1 (Y) $, the chain rule implies that $ z = F (y) $ evolves
according to $ \dot{z} = F ^\prime (y) f (y) $. Now, if $\Phi$ is a
numerical integrator, let $ \Phi _f \colon y _0 \mapsto y _1 $ denote
its application to the original system $ \dot{y} = f (y) $, and let
$ \Phi _g \colon ( y _0 , z _0 ) \mapsto ( y _1, z _1 ) $ denote its
application to the augmented system
\begin{equation}
  \label{eqn:augmented}
  \dot{y} = f (y) , \qquad \dot{z} = F ^\prime (y) f (y) ,
\end{equation}
corresponding to the vector field
$ g ( y , z ) = \bigl( f (y) , F ^\prime (y) f (y) \bigr) $. We say
that $\Phi$ is \emph{$F$-functionally equivariant} if $ \Phi _g $
preserves the relation $ z = F (y) $, i.e.,
$ \Phi _g \colon \bigl( y _0, F ( y _0 ) \bigr) \mapsto \bigl( y _1 ,
F ( y _1 ) \bigr) $, for all vector fields $f$ on $Y$. In other words,
the following diagram commutes:
\begin{equation*}
  \begin{tikzcd}
    y _0 \ar[r, |->, "\Phi _f"] \ar[d, |->, "{(\mathrm{id}, F)}
    "'] & y _1 \ar[d, |->, "{(\mathrm{id}, F)}
    "]\\
    ( y _0, z _0 ) \ar[r, |->, "\Phi _g"] & ( y _1, z _1 ) \rlap{ .}
  \end{tikzcd}
\end{equation*}
This is weaker than equivariance in the usual sense, since the diagram
need only commute for \eqref{eqn:augmented}, not arbitrary
$(\mathrm{id}, F)$-related vector fields. Preserving invariants
becomes the special case where the augmented equation reads
$ \dot{z} = 0 $ and the integrator leaves $z$ constant.

We develop a theory of functional equivariance and show that it
provides a useful tool kit for understanding the behavior of
(especially affine and quadratic) observables, including local
conservation laws and multisymplecticity. The paper is organized as
follows:
\begin{itemize}
\item \cref{sec:fe} characterizes the functional equivariance of a
  large class of numerical integrators, including B-series methods,
  and explores some consequences for both conservative and
  non-conservative dynamical systems. The main result, \cref{thm:fe},
  shows that a method is functionally equivariant for a class of
  observables if and only if it preserves invariants in that class. In
  particular, all B-series methods are affine functionally
  equivariant, and those preserving quadratic invariants are quadratic
  functionally equivariant.

\item \cref{sec:conservation} applies this framework to local
  conservation laws for PDEs and spatially semidiscretized PDEs. In
  particular, affine/quadratic functionally equivariant
  numerical integrators are seen to preserve discrete-time local
  conservation laws for affine/quadratic observables.

\item \cref{sec:multisymplectic} applies this framework to the
  multisymplectic conservation law for canonical Hamiltonian PDEs and
  spatially semidiscretized PDEs. We show that multisymplectic
  semidiscretization in space, followed by a symplectic integrator in
  time, yields a multisymplectic method in spacetime. We also show
  that hybrid finite elements may be used for multisymplectic
  semidiscretization, generalizing the results of \citet{McSt2020} to
  time-evolution problems.

\item Finally, \cref{sec:additive_partitioned} extends the results
  from the class of methods considered in \cref{sec:fe} to additive
  and partitioned methods, including additive/partitioned Runge--Kutta
  methods and splitting/composition methods.
\end{itemize}
We remark that many of the results, particularly in \cref{sec:fe} and
\cref{sec:additive_partitioned}, are obtained using only the
equivariance properties of methods with respect to affine maps, rather
than representing them in terms of trees or Runge--Kutta tableaux. In
particular, \cref{thm:differentiation} gives a new, tree-free proof
that B-series methods are closed under differentiation, while
\cref{thm:differentiation_additive_partitioned} generalizes this to
additive and partitioned methods.

\section{Functional equivariance}
\label{sec:fe}

\subsection{Basic definitions and results}
\label{sec:basic}

Let $ \Phi $ be a one-step numerical integrator, whose application to
a vector field $ f \in \mathfrak{X} (Y) $ with time-step size
$\Delta t$ gives a map $ \Phi _{\Delta t,f} \colon Y \rightarrow Y $,
$ y _0 \mapsto y _1 $. All the methods we will consider have
$ \Phi _{ \Delta t, f } = \Phi _{ 1, \Delta t f } $, so it suffices to
consider integrator maps $ \Phi _f \coloneqq \Phi _{ 1, f } $ with
unit time step. When we refer to a numerical integrator, we mean the
entire collection of maps
$ \Phi = \bigl\{ \Phi _f : f \in \mathfrak{X} (Y) ,\ Y \text{ a Banach
  space} \bigr\} $.\footnote{For some methods, such as implicit
  Runge--Kutta methods, $ \Phi _{ \Delta t , f } (y) $ might only be
  defined for sufficiently small $ \Delta t $. Including such
  integrators requires only the minor modification of viewing
  $ \Phi _f $ as a partial function.}

\begin{remark}
  While this definition covers a large class of numerical integrators,
  including B-series methods, other classes of methods require
  additional data besides $f$ in order to define an integrator map,
  e.g., an additive decomposition of $f$ or a partitioning of $Y$. In
  \cref{sec:additive_partitioned}, we will discuss how the results of
  this section generalize to such methods, including
  additive/partitioned Runge--Kutta methods and splitting/composition
  methods.
\end{remark}

The main class of numerical integrators of this type that we will
consider are B-series methods, which \citet*{McMoMuVe2016} proved are
completely characterized by the following property of affine
equivariance.

\begin{definition}
  \label{def:ae}
  Given an affine map $ A \colon Y \rightarrow U $, a pair of vector
  fields $ f \in \mathfrak{X} (Y) $ and $ g \in \mathfrak{X} (U) $ is
  \emph{$A$-related} if $ A ^\prime \circ f = g \circ A $. A numerical
  integrator $\Phi$ is \emph{affine equivariant} if
  $ A \circ \Phi _f = \Phi _g \circ A $ for all $A$-related $f$ and
  $g$, all affine maps $A$, and all Banach spaces $Y$ and $U$.
\end{definition}

\begin{remark}
  \label{rmk:strong_vs_weak}
  This is consistent with the definition of affine equivariance in
  \citep{McMoMuVe2016}. We distinguish it from the weaker definition
  in \citet{MuVe2016}, where the condition above is required only for
  affine isomorphisms rather than all affine maps.
\end{remark}

\begin{definition}
  \label{def:fe}
  Given a G\^ateaux differentiable map $ F \colon Y \rightarrow Z $
  and $ f \in \mathfrak{X} (Y) $, define
  $ g \in \mathfrak{X} ( Y \times Z ) $ by
  $ g ( y, z ) = \bigl( f (y) , F ^\prime (y) f (y) \bigr) $. We say
  that a numerical integrator $\Phi$ is \emph{$F$-functionally
    equivariant} if
  $ (\mathrm{id}, F ) \circ \Phi _f = \Phi _g \circ ( \mathrm{id}, F )
  $ for all $ f \in \mathfrak{X} (Y) $. That is, if
  $ \Phi _f \colon y _0 \mapsto y _1 $, then
  $ \Phi _g \colon \bigl( y _0, F ( y _0 ) \bigr) \mapsto \bigl( y _1
  , F ( y _1 ) \bigr) $.  Given a class of maps $\mathcal{F}$, we say
  that $\Phi$ is \emph{$ \mathcal{F} $-functionally equivariant} if
  this holds for all $ F \in \mathcal{F} ( Y, Z ) $ and all Banach
  spaces $Y$ and $Z$.
\end{definition}

This is a slight generalization of the situation considered in the
introduction: $Z$ may now be any Banach space rather than
$ \mathbb{R} $, and $F$ is only required to be G\^ateaux
differentiable rather than $ C ^1 $. Note that
$g \in \mathfrak{X} ( Y \times Z ) $ is precisely the vector field
corresponding to the augmented system \eqref{eqn:augmented}.

\begin{example}[Runge--Kutta methods]
  \label{ex:rk}
An $s$-stage Runge--Kutta method has the form
\begin{align*}
  Y _i &= y _0 + \Delta t \sum _{ j = 1 } ^s a _{ i j } f ( Y _j ) , \qquad  i = 1, \ldots, s ,\\
  y _1 &= y _0 + \Delta t \sum _{ i = 1 } ^s b _i f ( Y _i ) ,
\end{align*}
where $ a _{ i j } $ and $ b _i $ are given coefficients defining the
method. When this method is applied to the augmented system
\eqref{eqn:augmented}, we augment the method by
\begin{equation*}
  z _1 = z _0 + \Delta t \sum _{ i = 1 } ^s b _i F ^\prime ( Y _i ) f ( Y _i ) .
\end{equation*}
Note that the internal stages $ Z _1, \ldots, Z _s $ are not needed,
since the augmented vector field depends only on $y$. Hence, for a
Runge--Kutta method, $F$-functional equivariance says that
\begin{equation*}
  F ( y _1 ) = F ( y _0 ) + \Delta t \sum _{ i = 1 } ^s b _i F ^\prime ( Y _i ) f ( Y _i ) .
\end{equation*}
In particular, if $F$ is an invariant of $f$, so that
$ F ^\prime (y) f (y) = 0 $ for all $y \in Y $, then the terms of this
sum vanish, and we get $ F ( y _1 ) = F ( y _0 ) $.
\end{example}

\begin{proposition}
  \label{prop:afe}
  Every affine equivariant method is affine functionally
  equivariant.
\end{proposition}

\begin{proof}
  If $ F \colon Y \rightarrow Z $ is an affine map, then so is
  $ (\mathrm{id}, F ) \colon Y \rightarrow Y \times Z $. Since the
  vector fields $f$ and $g$ in \cref{def:fe} are
  $ ( \mathrm{id}, F ) $-related, the conclusion follows by affine
  equivariance.
\end{proof}

\begin{remark}
  \label{rmk:aromatic}
  The converse is generally not true. For instance, the map
  $ y _0 \mapsto y _0 + ( f \operatorname{div} f ) ( y _0 ) $, which
  is defined for finite-dimensional $Y$, is seen to be affine
  functionally equivariant but is not affine equivariant except in the
  weaker sense mentioned in \cref{rmk:strong_vs_weak}. This is an
  example of an ``aromatic'' series that is not a B-series,
  cf.~\citet{MuVe2016}.
\end{remark}

Since we are concerned for the time being with affine equivariant
numerical integrators, it is natural to make the following assumptions
on $\mathcal{F}$. This includes cases where $\mathcal{F}$ contains not
only affine maps but also quadratic or higher-degree polynomial maps.

\begin{assumption}
  \label{assumption:affine}
  The class of maps $\mathcal{F}$ satisfies the following:
  \begin{itemize}
  \item $\mathcal{F} (Y, Y) $ contains the identity map for all $Y$;
  \item $ \mathcal{F} ( Y, Z ) $ is a vector space for all $Y$ and
    $Z$;
  \item $ \mathcal{F} $ is invariant under composition with affine
    maps, in the following sense: If $ A \colon Y \rightarrow U $ and
    $ B \colon V \rightarrow Z $ are affine and
    $ F \in \mathcal{F} ( U, V ) $, then
    $ B \circ F \circ A \in \mathcal{F} ( Y, Z ) $.
  \end{itemize} 
\end{assumption}

As noted in the introduction, preservation of invariants may be seen
as a special case of functional equivariance, so one might expect the
latter property to be strictly stronger than the former.  Perhaps
surprisingly, our first main result shows that they are equivalent.

\begin{theorem}
  \label{thm:fe}
  Let $\mathcal{F}$ satisfy \cref{assumption:affine}. A numerical
  integrator $\Phi$ preserves $\mathcal{F}$-invariants if and only if
  it is $\mathcal{F}$-functionally equivariant.
\end{theorem}

\begin{proof}
  $( \Rightarrow ) $ Suppose $\Phi$ preserves
  $\mathcal{F}$-invariants. Given $ F \in \mathcal{F} ( Y, Z ) $, it
  follows from \cref{assumption:affine} that
  $ G ( y, z ) = F (y) - z $ is in $ \mathcal{F} ( Y \times Z , Z )
  $. This is an invariant of the augmented vector field
  $g(y, z ) = \bigl( f (y) , F ^\prime (y) f (y) \bigr) $, since
  $ G ^\prime ( y, z ) g ( y, z ) = F ^\prime (y) f (y) - F ^\prime
  (y) f (y) = 0 $. Hence, preservation of $\mathcal{F}$-invariants
  implies $ \Phi _g \colon ( y _0, z _0 ) \mapsto ( y _1 , z _1 ) $
  satisfies $ G ( y _1, z _1 ) = G ( y _0 , z _0 ) $, i.e.,
  $ F ( y _1 ) - z _1 = F ( y _0 ) - z _0 $. In particular,
  $ z _0 = F ( y _0 ) $ implies $ z _1 = F ( y _1 ) $.

  $ ( \Leftarrow ) $ Conversely, suppose $\Phi$ is
  $\mathcal{F}$-functionally equivariant. If
  $F \in \mathcal{F} ( Y, Z ) $ is an invariant of
  $f \in \mathfrak{X} ( Y ) $, then the augmented vector field is
  $ g = ( f, 0 ) $, and $\mathcal{F}$-functional equivariance implies
  $ \Phi _g \colon \bigl( y _0, F ( y _0 ) \bigr) \mapsto \bigl( y _1
  , F ( y _1 ) \bigr) $. However, any constant functional
  $ y \mapsto c \in Z $ is also in $ \mathcal{F} ( Y , Z ) $ and has
  the same augmented vector field $g$, so
  $ \Phi _g \colon ( y _0, c ) \mapsto ( y _1 , c ) $ for all
  $ c \in Z $. Hence, $ F ( y _1 ) = F ( y _0 ) $.
\end{proof}

\begin{corollary}
  \label{cor:b-series}
  For B-series methods, the following statements hold:
  \begin{enumerate}[label=(\alph*)]
  \item Every B-series method is affine functionally equivariant.
  \item B-series methods preserving quadratic invariants (e.g.,
    Gauss--Legendre collocation methods) are quadratic
    functionally equivariant.\label{cor:b-series_quadratic}
  \item No B-series method is cubic functionally equivariant.
  \end{enumerate} 
\end{corollary}

\begin{proof}
  This follows since B-series methods are affine equivariant and none
  preserves arbitrary cubic invariants (\citet*{ChMu2007,IsQuTs2007}).
\end{proof}

\subsection{Strong equivariance vs.~functional equivariance}

There is a stronger notion of $F$-equivariance, based on a
straightforward generalization of \cref{def:ae} to nonlinear maps
$ F \colon Y \rightarrow U $. Two vector fields
$ f \in \mathfrak{X} (Y) $ and $ g \in \mathfrak{X} (U) $ are
$F$-related if $ F ^\prime (y) f (y) = ( g \circ F ) (y) $ for all
$ y \in Y $, and $\Phi$ is $F$-equivariant if
$ F \circ \Phi _f = \Phi _g \circ F $ whenever this is the case.

To illustrate the distinction with functional equivariance, we now
show that the implicit midpoint method is \emph{not} quadratic
equivariant in this stronger sense, although
\cref{cor:b-series}\ref{cor:b-series_quadratic} tells us that it is
quadratic functionally equivariant. Let
$ F \colon \mathbb{R} \rightarrow \mathbb{R} , y \mapsto y ^2 $, and
observe that the vector fields
\begin{equation*}
  f (y) = -y , \qquad g (u) = -2 u ,
\end{equation*}
are $F$-related. Applying the implicit midpoint method with time step
size $\Delta t = 1 $ gives
\begin{equation*}
  y _1 = \frac{ 1 }{ 3 } y _0 , \qquad u _1 = 0 .
\end{equation*}
Since $ u _1 \neq ( y _1 ) ^2 $ for $ y _0 \neq 0 $, the method is not
$F$-equivariant. On the other hand, applying the method to the
augmented equation $ \dot{z} = F ^\prime (y) f (y) = - 2 y ^2 $ with $ z _0 = ( y _0 ) ^2 $ gives
\begin{equation*}
  z _1 = ( y _0 ) ^2  - 2 \biggl( \frac{ y _0 + y _1 }{ 2 } \biggr) ^2  = ( y _0 ) ^2 - 2 \biggl( \frac{ y _0 + \frac{ 1 }{ 3 } y _0 }{ 2 }  \biggr) ^2 = \frac{ 1 }{ 9 } (y _0) ^2 = ( y _1 ) ^2 ,
\end{equation*}
which illustrates that the method is $F$-functionally equivariant.

Essentially, functional equivariance requires only that $\Phi$ commute
with \emph{particular} pairs of related vector fields, while strong
equivariance requires that it commute with \emph{all} such pairs.

\subsection{Affine equivariance and closure under differentiation}
\label{sec:differentiation}

In addition to invariants and observables that depend on $y$ itself,
we are often interested in those that depend on variations of $y$. We
say that $\eta$ is a variation of $y$ if
$ ( y, \eta ) \in Y \times Y $ satisfy
\begin{equation}
  \label{eqn:variational}
  \dot{y} = f (y) , \qquad \dot{ \eta } = f ^\prime (y) \eta ,
\end{equation}
whose flow is the derivative of the flow of
$ f \in \mathfrak{X} (Y) $. An especially important example is the
canonical symplectic form for Hamiltonian ODEs, which can be
understood as a quadratic invariant depending on two variations of
$y$.

\begin{definition}
  A numerical method $\Phi$ is said to be \emph{closed under
    differentiation} if the method applied to \eqref{eqn:variational}
  is the derivative of $ \Phi _f $, i.e.,
  $ ( y _1 , \eta _1 ) = \bigl( \Phi _f ( y _0 ), \Phi ^\prime _f ( y
  _0 ) \eta _0 \bigr) $.
\end{definition}

\citet{BoSc1994} showed that Runge--Kutta methods are closed under
differentiation, from which it follows that those preserving quadratic
invariants are symplectic integrators. The same argument can be
applied to B-series methods, where closure under differentiation can
be established by showing that it holds for all trees \citep[Theorem
VI.7.1]{HaLuWa2006}. Here, we present a new, tree-free proof that uses
only affine equivariance, and which will readily generalize to
additive and partitioned methods in \cref{sec:additive_partitioned}.

\begin{theorem}
  \label{thm:differentiation}
  Affine equivariant numerical integrators are closed under
  differentiation.
\end{theorem}

\begin{proof}
  Given $ f \in \mathfrak{X}  (Y) $, consider the system
  \begin{equation*}
    \dot{x} = f (x) , \qquad \dot{y} = f (y) ,
  \end{equation*}
  corresponding to $ f \times f \in \mathfrak{X} ( Y \times Y )
  $. Since $ f \times f $ is $A$-related to $f$, where $A$ is either
  of the projections $ ( x, y ) \mapsto x $ or $ ( x, y ) \mapsto y $,
  it follows that $ \Phi _{ f \times f } = \Phi _f \times \Phi _f
  $. Now, let $ \epsilon > 0 $ and take
  $ z = F(x,y) = ( x - y ) / \epsilon $, giving the augmented system
  \begin{equation*}
    \dot{x} = f (x) , \qquad \dot{y} = f (y) , \qquad \dot{z} = \frac{ f (x) - f (y) }{ \epsilon } .
  \end{equation*}
  By \cref{prop:afe}, applying $\Phi$ to this system yields
  \begin{equation*}
    x _1 = \Phi _f ( x _0 ) , \qquad y _1 = \Phi _f ( y _0 ) , \qquad z _1 = \frac{ \Phi _f ( x _0 ) - \Phi _f ( y _0 ) }{ \epsilon } .
  \end{equation*}
  Finally, let $ x _0 = y _0 + \epsilon \eta _0 $ and take the limit
  as $ \epsilon \rightarrow 0 $.
\end{proof}

\begin{corollary}
  \label{cor:fe_variation}
  Let $\Phi$ be an affine equivariant numerical integrator preserving
  $\mathcal{F}$-invariants. Given
  $ F \in \mathcal{F} ( Y \times Y , Z ) $, define
  $ g \in \mathfrak{X} ( Y \times Y \times Z ) $ by
  \begin{equation*}
    g ( y, \eta, z ) = \Bigl( f (y) , f ^\prime (y) \eta, F ^\prime (
    y, \eta ) \bigl( f (y) , f ^\prime (y) \eta \bigr) \Bigr).
  \end{equation*}
  Then
  $ \Phi _g \bigl( y _0 , \eta _0, F ( y _0 , \eta _0 ) \bigr) =
  \bigl( y _1 , \eta _1 , F ( y _1, \eta _1 ) \bigr) $, where
  $ y _1 = \Phi _f ( y _0 ) $ and
  $ \eta _1 = \Phi _f ^\prime ( y _0 ) \eta _0 $.
\end{corollary}

\begin{proof}
  Apply \cref{thm:fe} and \cref{thm:differentiation}.
\end{proof}

\begin{remark}
  \label{rmk:multiple_variations}
  It is trivial to extend \cref{cor:fe_variation} to the case where
  $F$ depends on two or more variations of $y$, e.g.,
  $ F = F ( y, \xi, \eta ) $ where $ \dot{ \xi } = f ^\prime (y) \xi $
  and $ \dot{ \eta } = f ^\prime (y) \eta $.
\end{remark}

\subsection{Examples}

Before discussing applications to conservation laws for PDEs, which
will be the subject of \cref{sec:conservation}, we first illustrate
some examples of functional equivariance for numerical ODEs.

\subsubsection{Hamiltonian systems}

Suppose $Y$ is equipped with a Poisson bracket
$ \{ \cdot , \cdot \} $. Given $ H \colon Y \rightarrow \mathbb{R} $,
the corresponding Hamiltonian vector field $ f \in \mathfrak{X} (Y) $
is determined by the condition $ \dot{ F } = \{ F , H \} $ for
$ F \colon Y \rightarrow \mathbb{R} $. That is, the augmented system
\eqref{eqn:augmented} can be written
\begin{equation*}
  \dot{y} = f (y) , \qquad \dot{z} = \{ F, H \} (y) .
\end{equation*}
Hence, if $\Phi$ is $\mathcal{F}$-functionally equivariant, then
applying $\Phi$ to this system gives a ``discrete version'' of
$ \dot{ F } = \{ F, H \} $ for
$ F \in \mathcal{F} ( Y, \mathbb{R} ) $. For a Runge--Kutta method,
this has the form
\begin{equation*}
  F ( y _1 ) = F ( y _0 ) + \Delta t \sum _{ i = 1 } ^s b _i \{ F, H \} ( Y _i ) .
\end{equation*}
This holds for \emph{any} Poisson bracket, not just the canonical
bracket or those for which the Poisson tensor is
constant. Preservation of $\mathcal{F}$-invariants is the special case
where $ \{ F, H \} = 0 $.

\subsubsection{Canonical Hamiltonian systems with damping/forcing}
\label{sec:damping}

Let $ Y = \mathbb{R} ^{ 2 n } $ with canonical coordinates
$ y = ( q, p ) $. Consider the
system
\begin{equation}
  \label{eqn:hamiltonian_linear_damping}
  \dot{q} = \nabla _p H (q,p) , \qquad \dot{p} = - \nabla _q H (q,p) - c p ,
\end{equation}
where $H$ is a Hamiltonian and $ c \geq 0 $ a constant parameter. If
$H$ has the special form
$ H ( q, p ) = \frac{1}{2} p ^T M ^{-1} p + V (q) $, where $M$ is a
positive definite mass matrix and $V$ is a potential energy function,
then energy is dissipated according to
$ \frac{\mathrm{d}}{\mathrm{d}t} H ( q, p ) = - c p ^T M ^{-1} p $, and
the parameter $c$ dictates the rate of dissipation.

If $V$ is also quadratic, then so is $H$, and hence any quadratic
functionally equivariant method $\Phi$ yields a discrete version of
this dissipation law. If $\Phi$ is a Runge--Kutta method preserving
quadratic invariants, then this has the form
\begin{equation*}
  H ( q _1 , p _1 ) = H ( q _0 , p _0 ) - \Delta t \sum _{ i = 1 } ^s b _i c P _i ^T M ^{-1} P _i .
\end{equation*}
There is no reason to restrict to linear damping: if we replace the
damping term $ - c p $ in \eqref{eqn:hamiltonian_linear_damping} by an
arbitrary forcing term $ \phi ( q, p ) $, then we obtain
\begin{equation}
  \label{eqn:forcing}
  H(q _1, p _1 ) = H(q _0, p _0) + \Delta t \sum _{ i = 1 } ^s b _i P _i ^T M ^{-1} \phi ( Q _i  , P _i ) ,
\end{equation}
where the sum on the right-hand side approximates the work done by
$\phi$.  When $V$ is not quadratic, the identities above generally do
not hold. However, since the kinetic energy functional
$ \frac{1}{2} p ^T M p $ is quadratic, we still get the weaker
identity
\begin{equation*}
  \frac{1}{2} p _1 ^T M ^{-1} p _1 = \frac{1}{2} p _0 ^T M ^{-1} p _0 + \Delta t \sum _{ i = 1 } ^s b _i P _i ^T M ^{-1} \bigl[ - \nabla V ( Q _i ) + \phi ( Q _i , P _i ) \bigr] ,
\end{equation*} 
where the sum approximates work done by both conservative and
non-conservative forces.

\subsubsection{Monotone observables} Suppose
$ F \in \mathcal{F} (Y, \mathbb{R} ) $ is such that
$ F ^\prime (y) f (y) \leq 0 $, so $ F (y) $ is monotone
decreasing. If $\Phi$ is an $\mathcal{F}$-functionally equivariant
Runge--Kutta method with $ b _i \geq 0 $, then
\begin{equation*}
  F ( y _1 ) = F ( y _0 ) + \Delta t \sum _{ i = 1 } ^s b _i F ^\prime ( Y _i ) f ( Y _i ) \leq F ( y _0 ) ,
\end{equation*}
so $F$ is also monotone decreasing along the numerical
solution. Conversely, any method with this monotonicity property also
preserves $\mathcal{F}$-invariants, and is thus
$\mathcal{F}$-functionally equivariant, since $F$ is an invariant when
$ \pm F $ are both monotone decreasing.

\begin{remark}
  For Runge--Kutta methods, the additional condition $ b _i \geq 0 $
  is needed to get monotonicity. Functional equivariance alone is not
  sufficient. We are not aware of a more general version of this
  condition for arbitrary B-series methods.
\end{remark}

An immediate consequence is the known B-stability of Runge--Kutta
methods preserving quadratic invariants with $ b _i \geq 0 $. If $Y$
is a Hilbert space with inner product
$ \langle \cdot , \cdot \rangle $, consider
\begin{equation*}
  \dot{x} = f (x) , \qquad \dot{y} = f (y) ,
\end{equation*}
on $ Y \times Y $, and let
$ F ( x, y ) = \frac{1}{2} \lVert x - y \rVert ^2 $. Then
$ F ^\prime ( x, y ) \bigl( f (x) , f (y) \bigr) = \bigl\langle x - y
, f (x) - f (y) \bigr\rangle \leq 0 $ implies
$ F (x _1 , y _1 ) \leq F ( x _0 , y _0 ) $, i.e.,
$ \lVert x _1 - y _1 \rVert \leq \lVert x _0 - y _0 \rVert $. This is
precisely the condition for B-stability,
cf.~\citet{Butcher1975,BuBu1979}.

Another immediate application is to the dissipative systems in
\cref{sec:damping}, when $H$ is quadratic. If $ \phi $ is a
dissipative force, in the sense that
$ p ^T M ^{-1} \phi ( q, p ) \leq 0 $ for all
$ ( q, p ) \in \mathbb{R} ^{ 2 n } $, then \eqref{eqn:forcing} implies
$ H ( q _1 , p _1 ) \leq H ( q _0 , p _0 ) $, i.e., the quadratic
energy is monotone decreasing along the numerical solution.

\subsubsection{Symplectic and conformal symplectic systems}
\label{sec:symplectic}

Suppose that $\omega$ is a continuous bilinear form on $Y$. Let $\xi$
and $\eta$ each be variations of $y$, so that
$ ( y, \xi, \eta ) \in Y \times Y \times Y $ satisfy
\begin{equation*}
  \dot{y} = f (y) , \qquad \dot{ \xi } = f ^\prime (y) \xi , \qquad \dot{ \eta } = f ^\prime (y) \eta .
\end{equation*}
Then $ \omega ( \xi, \eta ) $ is a quadratic functional of this
augmented system, evolving according to
\begin{equation*}
  \frac{\mathrm{d}}{\mathrm{d}t} \omega ( \xi, \eta ) = \omega \bigl( f ^\prime (y) \xi, \eta \bigr) + \omega \bigl( \xi, f ^\prime (y) \eta \bigr) = (L _f \omega ) _y ( \xi, \eta ) ,
\end{equation*}
where $ (L _f \omega) _y $ is the Lie derivative along $f$ of $\omega$
at $y$ \citep[Theorem 6.4.1]{AbMaRa1988}.

If $\Phi$ preserves quadratic invariants, then we may apply quadratic
functional equivariance to describe the numerical evolution of
$\omega$. Taking
$ g \in \mathfrak{X} ( Y \times Y \times Y \times \mathbb{R} ) $ to be
\begin{equation*}
  g ( y, \xi, \eta, z ) = \bigl( f (y) , f ^\prime (y) \xi, f ^\prime (y) \eta, ( L _f \omega ) _y ( \xi, \eta ) \bigr) ,
\end{equation*}
it follows from \cref{cor:fe_variation,rmk:multiple_variations} that
\begin{equation}
  \label{eqn:symplectic}
  \Phi _g ( y _0 , \xi _0 , \eta _0, \omega ( \xi _0, \eta _0 ) \bigr) = \bigl( y _1, \xi _1 , \eta _1 , \omega ( \xi _1, \eta _1 ) \bigr) ,
\end{equation}
where $ y _1 = \Phi _f ( y _0 ) $,
$ \xi _1 = \Phi _f ^\prime ( y _0 ) \xi _0 $, and
$ \eta _1 = \Phi _f ^\prime ( y _0 ) \eta _0 $. Furthermore, this implies
\begin{equation*}
  \omega ( \xi _1 , \eta _1 ) = \omega \bigl( \Phi _f ^\prime ( y _0 ) \xi _0, \Phi _f ^\prime ( y _0 ) \eta _0 \bigr) = ( \Phi _f ^\ast \omega ) _{y _0} ( \xi _0, \eta _0 ) ,
\end{equation*}
where $ ( \Phi _f ^\ast \omega ) _{ y _0 } $ is the pullback of
$\omega$ by $ \Phi _f $ at $ y _0 $.  For a Runge--Kutta method
preserving quadratic invariants, \eqref{eqn:symplectic} takes the form
\begin{equation}
  \label{eqn:rk_symplectic}
  \omega  (\xi _1 , \eta _1 ) = \omega ( \xi _0 , \eta _0 ) + \Delta t \sum _{ i = 1 } ^s b _i ( L _f \omega ) _{ Y _i } ( \Xi _i , \mathrm{H} _i ) ,
\end{equation}
and the sum on the right-hand side expresses the difference between
$ \Phi _{\Delta t f} ^\ast \omega $ and $ \omega $ at $ y _0 $.

In particular, suppose that $\omega$ is antisymmetric and
nondegenerate, so that $ ( Y , \omega ) $ is a symplectic vector
space. If $f$ is a symplectic vector field, satisfying
$ L _f \omega = 0 $, then we recover the result of \citet{BoSc1994}
that if $\Phi$ preserves quadratic invariants, then
$ \Phi _f ^\ast \omega = \omega $, i.e., $\Phi$ is a symplectic
integrator. An interesting generalization is the case of
\emph{conformal symplectic} vector fields, satisfying
$ L _f \omega = - c \omega $ for some constant $c$, of which
\eqref{eqn:hamiltonian_linear_damping} is a canonical example; see
\citet{McPe2001}. In this case, \eqref{eqn:rk_symplectic} becomes
\begin{equation*}
  \omega  (\xi _1 , \eta _1 ) = \omega ( \xi _0 , \eta _0 ) - \Delta t \sum _{ i = 1 } ^s b _i c \omega ( \Xi _i , \mathrm{H} _i ) ,
\end{equation*}
which can be seen as an approximate conformal symplecticity
relation. However, $ \Phi _{ \Delta t f } ^\ast \omega $ generally
does not equal $ e ^{ - c \Delta t } \omega $ exactly unless
$ c = 0 $; see \citet[Example 7]{McQu2001} for a counterexample when
$\Phi$ is the implicit midpoint method.

\begin{remark}
  \label{rmk:vector_bilinear}
  The arguments above apply without modification if $\omega$ is a
  vector-valued bilinear form, i.e., a continuous bilinear map
  $ Y \times Y \rightarrow Z $ for some Banach space $Z$.
\end{remark}

\section{Application to conservation laws in PDEs}
\label{sec:conservation}

In this section, we apply the general results of \cref{sec:fe} to
local conservation laws in time-evolution PDEs. We also consider
discrete conservation laws in numerical PDEs, when semidiscretization
in space is combined with numerical integration in time.

\subsection{General approach and examples}
\label{sec:conservation_general}

Let $ \dot{y} = f (y) $ correspond to a time-dependent system of PDEs
on a domain $\Omega$, where the Banach space $Y$ is a function space
(or product of function spaces) on $\Omega$. Suppose that solutions
satisfy a local conservation law, in the form
\begin{equation}
  \label{eqn:claw_diff}
  \dot{ \rho } = - \operatorname{div} J ,
\end{equation}
where $\rho$ and $J$ depend on $y$. The notation is deliberately
suggestive of Maxwell's equations, where $\rho$ is charge density, $J$
is current density, and \eqref{eqn:claw_diff} is local conservation of
charge.

From \cref{thm:fe} and \cref{cor:b-series}, we immediately obtain a
powerful general statement about preservation of local conservation
laws under numerical integration.  If $ \rho = F (y) $, where
$ F \in \mathcal{F} ( Y , Z ) $ and $Z$ is an appropriate space of
densities, then $ F ^\prime (y) f (y) = - \operatorname{div} J (y) $,
and thus an $\mathcal{F}$-functionally equivariant integrator
satisfies a discrete-time version of \eqref{eqn:claw_diff}. For
instance, a Runge--Kutta method preserving $\mathcal{F}$-invariants
satisfies
\begin{equation*}
  \rho _1 = \rho _0 - \Delta t \sum _{ i = 1 } ^s b _i \operatorname{div} J ( Y _i ) .
\end{equation*}
We note that, while $\rho$ is required to be related to $y$ by a
functional in $\mathcal{F}$ (e.g., $\rho$ is affine or quadratic in
$y$), no such restriction is placed on $J$.  In particular, all
B-series methods preserve affine local conservation laws, while those
preserving quadratic invariants also preserve quadratic local
conservation laws. In the case of symplectic Runge--Kutta methods,
\citet[Section 3.1]{FrHy2021} recently proved this by a direct
computation, whereas here it is seen as a particular instance of
quadratic functional equivariance.

In addition to the differential form of the conservation law
\eqref{eqn:claw_diff}, we may also integrate over a compact subdomain
$ K \subset \Omega $ and apply the divergence theorem to get
\begin{equation}
  \label{eqn:claw_int}
  \frac{\mathrm{d}}{\mathrm{d}t} \int _K \rho = - \int _{ \partial K } J \cdot \widehat{ n } ,
\end{equation}
where $ \widehat{ n } $ denotes the outer unit normal to $K$. This may
be seen as an integral form of the conservation law
\eqref{eqn:claw_diff}. In this case, if $ \int _K \rho = F (y) $ with
$ F \in \mathcal{F} ( Y, \mathbb{R} ) $, then an
$\mathcal{F}$-functionally equivariant method satisfies a
discrete-time version of \eqref{eqn:claw_int}. In the case of a
Runge--Kutta method preserving $ \mathcal{F} $-invariants, this has
the form
\begin{equation*}
  \int _K \rho _1 = \int _K \rho _0 - \Delta t \sum _{ i = 1 } ^s b _i \int _{ \partial K } J ( Y _i ) \cdot \widehat{ n } .
\end{equation*}

\begin{example}
  Maxwell's equations in $\mathbb{R}^3$ consist of the vector
  evolution equations
  \begin{equation*}
    \label{eqn:maxwell_evolution}
    \dot{ B } = - \operatorname{curl} E , \qquad \dot{ {D} } = \operatorname{curl} H - J ,
  \end{equation*}
  along with the scalar constraint equations
  \begin{equation*}
    \operatorname{div} B = 0 , \qquad \operatorname{div} {D} = \rho ,
  \end{equation*}
  and the constitutive relations $ {D} = \epsilon E $ and
  $ B = \mu H $. Here, $E$ and $H$ are the electric and magnetic
  fields, $D$ and $B$ are the electric and magnetic flux densities,
  $\epsilon$ and $\mu$ are the electric permittivity and magnetic
  permeability tensors, and $\rho$ and $J$ are charge and current
  density.

  Taking the divergence of the first evolution equation, we see that
  $ \operatorname{div} B $ is a local invariant, so the constraint
  $ \operatorname{div} B = 0 $ is preserved by the evolution. Next,
  interpreting the constraint $ \operatorname{div} {D} = \rho $ to
  define $\rho$ as a function of $D$, we see that taking the
  divergence of the second evolution equation gives the local
  conservation law $ \dot{ \rho } = - \operatorname{div} J $.  Since
  $ \operatorname{div} B $ and $ \operatorname{div} {D} $ are both
  linear in $ y = ( B, {D} ) $, any B-series method will preserve the
  constraint $ \operatorname{div} B = 0 $, together with a discrete
  version of the conservation law relating $\rho$ and $J$.
\end{example}

\begin{example}
  \label{ex:nls}
  Consider the nonlinear Schr\"odinger (NLS) equation,
  \begin{equation*}
    i \dot{ u } +  \Delta u = \phi \bigl( \lvert u \rvert ^2 \bigr) u .
  \end{equation*}
  A direct computation shows that solutions satisfy
  \begin{equation*}
    \frac{ \partial }{ \partial t }  \frac{1}{2} \lvert u \rvert ^2
    = \Im ( \bar{ u } \,i \dot{u}  )\\
    = \Im ( - \bar{ u } \Delta u )\\
    = \Im \bigl(  - \operatorname{div} ( \bar{ u } \operatorname{grad} u ) \bigr)\\
    = - \operatorname{div} \Im ( \bar{ u } \operatorname{grad} u ) ,
  \end{equation*}
  which is a local conservation law for the quadratic functional
  $ F (u) = \frac{1}{2} \lvert u \rvert ^2 $. Since the implicit
  midpoint method is quadratic functionally equivariant, it follows
  that this conservation law is preserved, in the sense that
  \begin{equation*}
    \frac{1}{2} \lvert u _1 \rvert ^2 = \frac{1}{2} \lvert u _0 \rvert ^2 - \Delta t \operatorname{div} \Im \biggl( \frac{ \bar{ u } _0 + \bar{ u } _1 }{ 2 } \operatorname{grad} \frac{ u _0 + u _1 }{2} \biggr).
  \end{equation*}
  More generally, for any quadratic functionally equivariant
  Runge--Kutta method,
  \begin{equation*}
    \frac{1}{2} \lvert u _1 \rvert ^2 = \frac{1}{2} \lvert u _0 \rvert ^2 - \Delta t \sum _{ i = 1 } ^s b _i \operatorname{div} \Im ( \bar{ U } _i \operatorname{grad} U _i ) .
  \end{equation*}
\end{example}

\begin{example}
  \label{ex:wave}
  Consider the wave equation $ \ddot{ u } = \Delta u $, written as the
  first-order system
  \begin{equation*}
    \dot{ u } = p , \qquad \dot{ p } = \Delta u .
  \end{equation*}
  If $ y = ( u, p ) $ is a solution, then
  \begin{equation*}
    \frac{ \partial }{ \partial t } \frac{1}{2} \bigl( p ^2 + \lvert \operatorname{grad} u \rvert ^2 \bigr) = p \dot{ p } + \operatorname{grad} u \cdot \operatorname{grad} \dot{ u } = p \Delta u + \operatorname{grad} u \cdot \operatorname{grad} p = \operatorname{div} ( p \operatorname{grad} u ) ,
  \end{equation*}
  which is a local conservation law for the quadratic functional
  $ F (u, p) = \frac{1}{2} \bigl( p ^2 + \lvert \operatorname{grad} u
  \rvert ^2 \bigr) $. Hence, applying any quadratic functionally
  equivariant Runge--Kutta method gives
  \begin{equation*}
    \frac{1}{2} \bigl( (p _1) ^2 + \lvert \operatorname{grad} u _1 \rvert ^2 \bigr) = \frac{1}{2} \bigl( (p _0) ^2 + \lvert \operatorname{grad} u _0 \rvert ^2 \bigr) + \Delta t \sum _{ i = 1 } ^s b _i \operatorname{div} ( P _i \operatorname{grad} U _i ) .
  \end{equation*}
\end{example}

\subsection{Discrete conservation laws in numerical PDEs} In practice,
of course, we will not be applying numerical integrators to
infinite-dimensional function spaces. Rather, we typically first
semidiscretize in space (e.g., using a finite difference, finite
volume, or finite element scheme), yielding a system
$ \dot{ y } _h = f _h ( y _h ) $ on a finite-dimensional vector space
$ Y _h $ to which we can apply a numerical integrator.

Suppose the spatial discretization scheme is such that it preserves a
semidiscrete conservation law
$ \dot{ \rho } _h = - \operatorname{div} _h J _h $, where
$ \operatorname{div} _h $ is some ``discrete divergence'' operator.
Then it follows from the previous section that, if
$ \rho _h = F _h ( y _h ) $ for some
$ F _h = \mathcal{F} ( Y _h , Z _h ) $, then applying an
$\mathcal{F}$-functionally equivariant numerical integrator yields a
fully discrete conservation law corresponding to
\eqref{eqn:claw_diff}. We illustrate this with a few examples, which
are semidiscretized versions of those considered in the previous
section.

\begin{example}
  \citet{Nedelec1980} introduced a finite element semidiscretization
  of Maxwell's equations, in which $E$ and $B$ are approximated by
  piecewise-polynomial vector fields
  $ E _h \in H ( \operatorname{curl} ; \Omega ) $ and
  $ B _h \in H ( \operatorname{div}; \Omega ) $. This method may be written as
  \begin{equation*}
    \dot{ B } _h = - \operatorname{curl} E _h , \qquad \int _\Omega \dot{ {D}} _h \cdot v _h = \int _\Omega ( H _h \cdot \operatorname{curl} v _h - J _h \cdot v _h ) ,
  \end{equation*}
  where $ {D} _h = \epsilon E _h $, $ H _h = \mu ^{-1} B _h $, and
  $ v _h $ is any vector field from the same space as $ E _h $.

  Taking the divergence of the first equation gives
  $ \operatorname{div} \dot{ B } _h = 0 $, so the constraint
  $ \operatorname{div} B _h = 0 $ is preserved by the evolution. For
  the second, when $ v _h = -\operatorname{grad} \phi _h $ for a
  piecewise-polynomial scalar field $ \phi _h $, we get
  \begin{equation*}
    -\int _\Omega \dot{ {D} } _h \cdot \operatorname{grad} \phi _h = \int _\Omega J _h \cdot \operatorname{grad} \phi _h ,
  \end{equation*}
  which we may write as
  $ \operatorname{div} _h \dot{ {D} } _h = - \operatorname{div} _h J
  _h $. Thus, taking $ \rho _h = \operatorname{div} _h {D} _h $
  implies the semidiscrete charge conservation law
  $ \dot{ \rho } _h = - \operatorname{div} _h J _h $. (See
  \citet{BeSt2021} for a hybridization of \citeauthor{Nedelec1980}'s
  method that preserves a stronger form of this conservation law,
  using $ \operatorname{div} $ rather than $ \operatorname{div} _h $.)
  Since $ \operatorname{div} B _h $ and
  $ \operatorname{div} _h {D} _h $ are linear in
  $ y _h = ( B _h , {D} _h ) $, any B-series method will preserve
  $ \operatorname{div} B _h = 0 $ exactly and give a discrete-time
  version of the charge conservation law relating $ \rho _h $ and
  $ J _h $.
\end{example}

\begin{example}
  For the one-dimensional NLS equation, the finite-difference
  semidiscretization
  \begin{equation*}
    i \dot{ u } _k + \frac{ u _{ k + 1 } - 2 u _k + u _{ k -1 } }{ h ^2 } = \phi \bigl( \lvert u _k \rvert ^2 \bigr) u _k 
  \end{equation*}
  satisfies the semidiscrete local conservation law
  \begin{equation*}
    \frac{\mathrm{d}}{\mathrm{d}t} \frac{1}{2} \lvert u _k \rvert ^2 = - \frac{ 1 }{ h } \Biggl[ \Im \Biggl( \biggl(\frac{ \bar{u} _k + \bar{u} _{ k + 1 } }{ 2 } \biggr) \biggl( \frac{ u _{ k + 1 } - u _k }{ h } \biggr) \Biggr) - \Im \Biggl( \biggl( \frac{ \bar{u} _{k-1} + \bar{u} _k }{ 2 } \biggr) \biggl( \frac{ u _k - u _{k-1} }{ h } \biggr) \Biggr) \Biggr],
  \end{equation*}
  where the right-hand side is a difference of midpoint approximations
  to $ \Im ( \bar{u} \,\partial _x u ) $. Hence, a discrete-time
  version of this conservation law is preserved by any B-series method
  that preserves quadratic invariants.
\end{example}

\begin{example}
  \label{ex:nls_discrete}
  For the one-dimensional wave equation, consider the
  finite-difference semidiscretization
  \begin{equation*}
    \dot{ u } _k = p _k , \qquad \dot{ p } _k = \frac{ u _{ k + 1 } - 2 u _k + u _{ k -1 } }{ h ^2 } .
  \end{equation*}
  If we define
  \begin{equation*}
    \rho _k = \frac{1}{2} p _k ^2 + \frac{ 1 }{ 4 } \biggl( \frac{ u _{ k + 1 } - u _k }{ h } \biggr) ^2 + \frac{ 1 }{ 4 } \biggl( \frac{ u _k - u _{k-1} }{ h } \biggr) ^2 ,
  \end{equation*} 
  which is a finite-difference approximation to
  $ \frac{1}{2} \bigl( p ^2 + ( \partial _x u ) ^2 \bigr) $, then a
  short calculation gives the semidiscrete conservation law
  \begin{equation*}
    \dot{ \rho } _k = \frac{ 1 }{ h } \Biggl[ \biggl( \frac{ p _k + p _{ k + 1 } }{ 2 } \biggr) \biggl( \frac{ u _{ k + 1 } - u _k }{ h } \biggr) -  \biggl( \frac{ p _{ k -1 } + p _k }{ 2 } \biggr) \biggl( \frac{ u _k - u _{ k -1 } }{ h } \biggr) \Biggr],
  \end{equation*}
  where the right-hand side is a difference of midpoint approximations
  to $ p \, \partial _x u $. As in the previous example, a
  discrete-time version of this conservation law is therefore
  preserved by any B-series method that preserves quadratic
  invariants.
\end{example}

\subsection{Remarks on quadratic conservation laws arising from point
  symmetries}

Conservation laws with quadratic densities are common in partial
differential and differential-difference equations because of their
association with linear symmetries of Hamiltonian PDEs. (See, e.g.,
\citet{Olver1993}.) However, not all such symmetries are easily
preserved under semidiscretization. We focus here on affine point
symmetries, those arising from actions on the field variables.

For example, the one-dimensional NLS equation may be written in the
form
\begin{equation*}
  i \dot{ u } = \frac{ \delta }{ \delta \bar{ u } } \int _\Omega \Bigl( \lvert \partial _x u \rvert ^2 + V \bigl( \lvert u \rvert ^2 \bigr) \Bigr),
\end{equation*}
where $ V ^\prime = \phi $ and $ \delta / \delta \bar{u} $ is the
variational derivative with respect to $ \bar{u} $. The integrand
$ \mathcal{H} = \lvert \partial _x u \rvert ^2 + V \bigl( \lvert u
\rvert ^2 \bigr) $ is called the \emph{Hamiltonian density}. Observe
that $\mathcal{H}$ is invariant under the diagonal $ U (1) $ action
$ (u, \partial _x u ) \mapsto ( e ^{ i \alpha } u , e ^{ i \alpha }
\partial _x u ) $, where
$ \alpha \in \mathbb{R} \cong \mathfrak{u} (1) $. This point symmetry
leads to the local conservation law for
$ \rho = \frac{1}{2} \lvert u \rvert ^2 $ in \cref{ex:nls}. More
generally, any Hamiltonian density of the form
$ \mathcal{H} = \mathcal{H} \bigl( \lvert u \rvert ^2 , \lvert
\partial _x u \rvert ^2 , \bar{ u } \,\partial _x u \bigr) $ has the
same point symmetry, and hence
$ i \dot{ u } = \frac{ \delta }{ \delta \bar{u} } \int _\Omega
\mathcal{H} $ has a local conservation law for
$ \rho = \frac{1}{2} \lvert u \rvert ^2 $.

Similarly, the one-dimensional semidiscretized NLS equation in
\cref{ex:nls_discrete} can be written
\begin{equation*}
  i \dot{ u } _k = \frac{ \partial }{ \partial \bar{u} _k } \sum _\ell \biggl( \Bigl\lvert \frac{ u _{ \ell + 1 } - u _\ell }{ h } \Bigr\rvert ^2 + \frac{ V \bigl( \lvert u _\ell \rvert ^2 \bigr) + V \bigl( \lvert u _{\ell+1} \rvert ^2 \bigr) }{ 2 } \biggr),
\end{equation*}
where the summand can be viewed as a discrete Hamiltonian density
$ \mathcal{H} _h $. The invariance of $ \mathcal{H} _h $ under the
point symmetry $ u _\ell \mapsto e ^{ i \alpha } u _\ell $,
$ u _{ \ell + 1 } \mapsto e ^{ i \alpha } u _{ \ell + 1 } $, yields
the semidiscrete local conservation law for
$ \rho _k = \frac{1}{2} \lvert u _k \rvert ^2 $ obtained in
\cref{ex:nls_discrete}. More generally, we get such a local
conservation law whenever the discrete Hamiltonian density has the
form
$ \mathcal{H} _h = \mathcal{H} _h \bigl( \lvert u _\ell \rvert ^2 ,
\lvert u _{ \ell + 1 } \rvert ^2 , \bar{ u } _\ell u _{ \ell + 1 }
\bigr) $.

A related example involves orthogonal (rather than unitary) point
symmetry. Suppose $u (t, x) $ and its conjugate momentum
$ p ( t, x ) $ both take values in $\mathbb{R}^3$, and let
$ A \in O (3) $ act by
$ z = ( u , p, \partial _x u , \partial _x p ) \mapsto ( A u, A p, A
\partial _x u , A \partial _x p ) $. Then any Hamiltonian density that
depends only on the 10 invariants $ z _i \cdot z _j $,
$ 1 \leq i \leq j \leq 4 $, is $ O (3) $ invariant and thus has a
local conservation law for $ \rho = u \times p $. Like the $ U (1) $
point symmetry discussed above, this $ O (3) $ point symmetry is
preserved under a wide class of lattice semidiscretizations, which
have corresponding semidiscrete quadratic conservation laws.

By contrast with point symmetries, symmetries that involve spatial
translations are typically broken by semidiscretization. However,
special semidiscretizations can be constructed that preserve versions
of the associated conservation laws, although these are generally not
symplectic. An example is provided by the Korteweg--de~Vries equation
\begin{equation*}
  \partial _t u = \partial _x ( \alpha u ^2 ) + \nu \partial _x ^3 u ,
\end{equation*}
which has a local conservation law with $ \rho = u ^2 $. The
semidiscretization
\begin{equation*}
  \dot{ u } _k = \frac{ \alpha }{ 2 h } \bigl[ \theta ( u _{ k + 1 } ^2 - u _{ k -1 } ^2 ) + 2 ( 1 - \theta ) u _k ( u _{ k + 1 } - u _{ k -1 } ) \bigr] + \frac{ \nu }{ 2 h ^3 } ( u _{ k + 2 } - 2 u _{ k + 1 } + 2 u _{ k -1 } - u _{ k - 2 } ) 
\end{equation*}
has a semidiscrete conservation law with density $ \rho _k = u _k ^2 $
only for the parameter $ \theta = 2/3 $ (\citet{AsMc2004}).

\citet{FrHy2021} give general techniques for constructing
finite-difference semidiscretizations preserving several local
conservation laws---linear, quadratic, or otherwise---with many
examples. When such methods are used in conjunction with B-series
methods for time integration, it follows from \cref{thm:fe} and
\cref{cor:b-series} that affine local conservation laws are always
preserved in a discrete sense, while quadratic local conservation laws
are preserved by any B-series method that preserves quadratic
invariants.

\section{Multisymplectic integrators}
\label{sec:multisymplectic}

In this section, we apply the foregoing theory to the multisymplectic
conservation law for canonical Hamiltonian PDEs and its preservation
by numerical integrators. Since this is a quadratic local conservation
law depending on variations of solutions, it follows that B-series
methods preserving quadratic invariants also preserve a discrete-time
version of the multisymplectic conservation law. Furthermore, we
discuss techniques for spatial semidiscretization that preserve a
semidiscrete multisymplectic conservation law, reviewing some known
results for finite-difference semidiscretization and introducing new
results for finite-element semidiscretization. Consequently, when such
methods are used in conjunction with B-series methods preserving
quadratic invariants, the resulting method will satisfy a fully
discrete multisymplectic conservation law.

\subsection{Canonical Hamiltonian PDEs}
Before discussing the canonical Hamiltonian formalism for
time-evolution PDEs, we first quickly recall the stationary
(time-independent) case, following the treatment in \citet{McSt2020}.

Given a spatial domain $ \Omega \subset \mathbb{R}^m $ with
coordinates $ x = ( x ^1 , \ldots, x ^m ) $, let
$ u \colon \Omega \rightarrow \mathbb{R}^n $ and
$ \sigma \colon \Omega \rightarrow \mathbb{R} ^{ m n } $ be unknown
fields. The \emph{de~Donder--Weyl equations}
\citep{deDonder1935,Weyl1935} for a Hamiltonian
$ H \colon \Omega \times \mathbb{R}^n \times \mathbb{R} ^{ m n }
\rightarrow \mathbb{R} $, $ H = H ( x, u, \sigma ) $, are
\begin{equation}
  \label{eqn:ddw_space}
  \partial _\mu u ^i = \frac{ \partial H }{ \partial \sigma _i ^\mu } , \qquad - \partial _\mu \sigma _i ^\mu = \frac{ \partial H }{ \partial u ^i } ,
\end{equation}
where $ \mu = 1 , \ldots, m $ and $ i = 1 , \ldots, n $. Here and
henceforth, we use the Einstein index convention of summing over
repeated indices; for instance, $ \partial _\mu \sigma _i ^\mu $ has
an implied sum over $\mu$ and therefore corresponds to the divergence
of the vector field $ \sigma _i $.

Now, for time-dependent problems, we let $u = u ^i ( t, x ) $ and
$ \sigma = \sigma _i ^\mu ( t, x ) $ depend on
$t \in ( t _0, t _1 ) $, and we introduce an additional unknown field
$ p = p _i ( t, x ) $. The de~Donder--Weyl equations for
$ H \colon ( t _0, t _1 ) \times \Omega \times \mathbb{R}^n \times
\mathbb{R}^n \times \mathbb{R} ^{ m n } $,
$ H = H ( t, x, u, p, \sigma ) $, are then given by
\begin{equation}
  \label{eqn:ddw_time}
  \dot{ u } ^i = \frac{ \partial H }{ \partial p _i } , \qquad \partial _\mu u ^i = \frac{ \partial H }{ \partial \sigma _i ^\mu } , \qquad - ( \dot{p} _i + \partial _\mu \sigma _i ^\mu ) = \frac{ \partial H }{ \partial u ^i } .
\end{equation}
Note that \eqref{eqn:ddw_time} is simply \eqref{eqn:ddw_space} in
$ (m + 1) $-dimensional spacetime, where we have adopted the special
notation $ t = x ^0 $ and $ p _i = \sigma _i ^0 $. Moreover, the
special case $ m = 0 $ recovers canonical Hamiltonian mechanics on
$ \mathbb{R} ^{ 2 n } $.

For $ m > 0 $, the de~Donder--Weyl equations are not in the form
$ \dot{ y } = f (y) $, since we have expressions for $ \dot{ u } $ and
$ \dot{ p } $ but not $ \dot{ \sigma } $. To deal with this, we assume
that the second equation of \eqref{eqn:ddw_time} defines $\sigma$ as
an implicit function of $t$, $x$, $u$, $p$, and
$ \operatorname{grad} u $. By the implicit function theorem, this is
true (at least locally) if the $ m n \times m n $ matrix
$ \partial ^2 H / ( \partial \sigma _i ^\mu \partial \sigma _j ^\nu) $
is nondegenerate. Therefore, we may eliminate the second equation and
substitute this expression for $\sigma$ into the other two
equations. Assuming the Hamiltonian does not depend on $t$, this gives
a system of the form $ \dot{y} = f (y) $ with $ y = ( u, p ) $.

\begin{example}
  Let $ n = 1 $, so that $u$ and $p$ are scalar fields and $\sigma$ is
  a vector field on $\Omega$, and take
  $ H = \frac{1}{2} \bigl( p ^2 - \lvert \sigma \rvert ^2 ) $. Then
  the de~Donder--Weyl equations are
  \begin{equation*}
    \dot{ u } = p , \qquad \operatorname{grad} u = - \sigma , \qquad - (\dot{ p } + \operatorname{div} \sigma) = 0 .
  \end{equation*}
  Eliminating the second equation and substituting
  $ \sigma = - \operatorname{grad} u $ into the third, we obtain the
  first-order form of the wave equation with $ y = ( u, p ) $, as in
  \cref{ex:wave}.
\end{example}

\subsection{The multisymplectic conservation law}

For Hamiltonian ODEs, the symplectic conservation law is a statement
about variations of solutions to Hamilton's equations. Similarly, for
Hamiltonian PDEs, the multisymplectic conservation law is a statement
about variations of solutions to the de~Donder--Weyl equations.

\begin{definition}
  Let $ ( u, p, \sigma ) $ be a solution to \eqref{eqn:ddw_time}. A
  \emph{(first) variation} of $ ( u, p , \sigma ) $ is a solution
  $ ( v, r, \tau ) $ to the linearized problem
  \begin{alignat*}{4}
    \dot{ v } ^i
    &{}={}& \frac{ \partial ^2 H }{ \partial p _i \partial u ^j } v ^j
    &{}+{}& \frac{ \partial ^2 H }{ \partial p _i \partial p _j } r _j
    &{}+{}& \frac{ \partial ^2 H }{ \partial p _i \partial \sigma _j ^\nu } \tau _j ^\nu ,\\
    \partial _\mu v ^i
    &{}={}& \frac{ \partial ^2 H }{ \partial \sigma _i ^\mu \partial u ^j } v ^j
    &{}+{}& \frac{ \partial ^2 H }{ \partial \sigma _i ^\mu \partial p _j } r _j
    &{}+{}& \frac{ \partial ^2 H }{ \partial \sigma _i ^\mu \partial \sigma _j ^\nu } \tau _j ^\nu ,\\
    - ( \dot{ r } _i + \partial _\mu \tau _i ^\mu )
    &{}={}& \frac{
      \partial ^2 H }{ \partial u ^i \partial u ^j } v ^j
    &{}+{}&
    \frac{ \partial ^2 H }{ \partial u ^i \partial p _j } r _j
    &{}+{}&
    \frac{ \partial ^2 H }{ \partial u ^i \partial \sigma _j ^\nu }
    \tau _j ^\nu ,
  \end{alignat*}
  where the Hessians on the right-hand side are evaluated at
  $ ( t, x, u, p, \sigma ) $.
\end{definition}

On the space
$ \mathbb{R} ^n \times \mathbb{R}^n \times \mathbb{R} ^{ m n } \ni (
u, p, \sigma ) $, we now define the canonical $2$-forms
$ \omega ^0 = \mathrm{d} u ^i \wedge \mathrm{d} p _i $ and
$ \omega ^\mu = \mathrm{d} u ^i \wedge \mathrm{d} \sigma _i ^\mu $ for
$ \mu = 1 , \ldots, m $. The \emph{multisymplectic conservation law}
states that, for any pair of variations $ ( v, r, \tau ) $ and
$ ( v ^\prime , r ^\prime , \tau ^\prime ) $, we have
\begin{equation*}
  \partial _t \Bigl( \omega ^0 \bigl( ( v, r, \tau ) , (v ^\prime , r ^\prime , \tau ^\prime ) \bigr) \Bigr) = - \partial _\mu \Bigl( \omega ^\mu \bigl( ( v, r, \tau ) , (v ^\prime , r ^\prime , \tau ^\prime ) \bigr) \Bigr) ,
\end{equation*}
that is,
\begin{equation*}
\partial _t ( v ^i r _i ^\prime - v ^{ \prime i } r _i ) = - \partial _\mu ( v ^i \tau _i ^{ \prime \mu } - v ^{ \prime i } \tau _i ^\mu ) .
\end{equation*}
The proof is simply a calculation, using the symmetry of the
Hessian. We abbreviate the multisymplectic conservation law as
\begin{equation}
  \label{eqn:mscl_diff}
  \dot{ \omega } ^0 = - \partial _\mu \omega ^\mu ,
\end{equation}
with the understanding that both sides are evaluated on variations of
solutions to \eqref{eqn:ddw_time}. In the special case $ m = 0 $, we
recover the usual symplectic conservation law for Hamiltonian ODEs.
As with the conservation laws in \cref{sec:conservation_general}, we
may also integrate \eqref{eqn:mscl_diff} over a compact subdomain
$ K \subset \Omega $ and apply the divergence theorem to get
\begin{equation}
  \label{eqn:mscl_int}
  \int _K \dot{ \omega } ^0 \,\mathrm{d}^m x = - \int _{ \partial K } \omega ^\mu \,\mathrm{d}^{m-1} x_\mu ,
\end{equation}
which is an integral form of the multisymplectic conservation
law. Here,
$ \mathrm{d} ^m x \coloneqq \mathrm{d} x ^1 \wedge \cdots \wedge
\mathrm{d} x ^m $ is the standard Euclidean volume form on
$ \mathbb{R}^m $ and
$ \,\mathrm{d}^{m-1} x_\mu \coloneqq \iota _{ e _\mu } \,\mathrm{d}^m
x $ is its interior product with the $\mu$th standard basis
vector. Again, we interpret \eqref{eqn:mscl_int} to mean that the
equality holds when both sides are evaluated on arbitrary variations
of solutions.

\begin{remark}
  \label{rmk:mscl_symplectic}
  If $\Omega$ is compact, and boundary conditions are chosen so that
  $ \omega ^\mu \,\mathrm{d}^{m-1} x_\mu = 0 $ on $ \partial \Omega $,
  then taking $ K = \Omega $ in \eqref{eqn:mscl_int} gives
  $ \int _\Omega \dot{ \omega } ^0 \,\mathrm{d}^m x = 0 $. This may be
  interpreted as invariance of the symplectic form
  $ \omega = \int _\Omega \omega ^0 \,\mathrm{d}^m x $ on the
  infinite-dimensional phase space $Y$.
\end{remark}

\subsection{Discrete-time multisymplectic conservation laws for
  numerical integrators}

As before, assume that the Hamiltonian does not depend on $t$ and that
$\sigma$ is an implicit function of the remaining variables, so that
\eqref{eqn:ddw_time} can be reduced to $ \dot{y} = f (y) $ with
$ y = ( u, p ) $. It follows that variations evolve according to
$ \dot{ \eta } = f ^\prime (y) \eta $ with $ \eta = ( v, r ) $, where
$\tau = \sigma ^\prime (y) \eta $. Hence, \eqref{eqn:mscl_diff} may be
seen as a quadratic conservation law involving variations of $y$, and
the results of \cref{sec:fe} immediately apply to give the following.

\begin{theorem}
  \label{thm:mscl}
  Suppose that \eqref{eqn:ddw_time} can be written as
  $ \dot{y} = f (y) $ with $ y = ( u , p ) $, where $\sigma$ is an
  implicit function of the other variables. Let $ \xi = ( v, r ) $ and
  $ \eta = ( v ^\prime , r ^\prime ) $, with
  $ \tau = \sigma ^\prime (y) \xi $ and
  $ \tau ^\prime = \sigma ^\prime (y) \eta $, and define the augmented
  vector field
  \begin{equation*}
    g ( y, \xi, \eta, z ) = \biggl( f (y) , f ^\prime (y) \xi, f ^\prime (y) \eta, - \partial _\mu \Bigl( \omega ^\mu \bigl( ( v, r, \tau  ), ( v ^\prime , r ^\prime , \tau ^\prime )  \bigr) \Bigr) \biggr) .
  \end{equation*}
  If $\Phi$ is an affine equivariant method preserving quadratic
  invariants, then
  \begin{equation*}
    \Phi _g \Bigl( y _0 , \xi _0, \eta _0, \omega ^0 \bigl( ( v _0, r _0, \tau _0 ) , ( v _0 ^\prime , r _0 ^\prime , \tau _0 ^\prime ) \bigr) \Bigr) = \Bigl( y _1 , \xi _1, \eta _1 , \omega ^0 \bigl( ( v _1 , r _1 , \tau _1 ) , ( v _1 ^\prime , r _1 ^\prime , \tau _1 ^\prime ) \bigr) \Bigr),
  \end{equation*}
  where $ y _1 = \Phi _f ( y _0 ) $,
  $ \xi _1 = \Phi _f ^\prime ( y _0 ) \xi _0 $, and
  $ \eta _1 = \Phi _f ^\prime ( y _0 ) \eta _0 $.
\end{theorem}

\begin{proof}
  The key observation is that
  $ \omega ^0 \bigl( ( v , r , \tau ) , ( v ^\prime , r ^\prime , \tau
  ^\prime ) \bigr) = v ^i r _i ^\prime - v ^{ \prime i } r _i $ is
  quadratic in $\xi$ and $\eta$ alone, so it is not affected by the
  (possibly nonlinear) dependence of $\sigma$ and its variations on
  the other variables. Hence, the result follows from
  \eqref{eqn:symplectic} and \cref{rmk:vector_bilinear}.
\end{proof}

\begin{corollary}
  For a Runge--Kutta method preserving quadratic invariants, we have
  \begin{multline*}
    \omega ^0 \bigl( ( v _1 , r _1, \tau _1 ) , ( v _1 ^\prime , r _1 ^\prime , \tau _1 ^\prime ) \bigr) \\
    = \omega ^0 \bigl( ( v _0, r _0, \tau _0 ) , ( v _0 ^\prime , r
    _0 ^\prime, \tau _0 ^\prime ) \bigr) - \Delta t \sum _{ i = 1 }
    ^s b _i \partial _\mu \Bigl( \omega ^\mu \bigl( ( V _i , R _i, T
    _i ), (V _i ^\prime, R _i ^\prime, T _i ^\prime ) \bigr) \Bigr).
  \end{multline*}
  This may be written equivalently as
  \begin{equation*}
    (\mathrm{d} u _1 ) ^j \wedge ( \mathrm{d} p _1  ) _j  = ( \mathrm{d} u _0 ) ^j \wedge ( \mathrm{d} p _0 ) _j  - \Delta t \sum _{ i = 1 } ^s b _i \partial _\mu \bigl(  ( \mathrm{d} U _i ) ^j \wedge ( \mathrm{d} \Sigma _i ) _j ^\mu \bigr) .
  \end{equation*}
\end{corollary}

\subsection{Multisymplectic semidiscretization on rectangular grids}

If $\Omega$ is a Cartesian product of intervals, equipped with a
rectangular finite-difference grid, there is a substantial literature
on spatial semidiscretization such that a semidiscrete multisymplectic
conservation law holds. We refer the reader in particular to the
following (non-exhaustive) list of references:
\citet*{Reich2000,BrRe2001,RyMc2008,McRySu2014}. These
semidiscretization schemes generally apply a symplectic Runge--Kutta
or partitioned Runge--Kutta method in each of the spatial
directions. In light of \cref{sec:fe}, the semidiscrete
multisymplectic conservation law may be seen as resulting from $m$
applications of \eqref{eqn:rk_symplectic} or its generalization to
partitioned methods in \cref{sec:additive_partitioned}.

In one dimension of space, \citet{SuXi2020} have recently investigated
multisymplectic semidiscretization using discontinuous Galerkin finite
element methods.

\subsection{Multisymplectic semidiscretization with hybrid finite
  element methods}

In \citet{McSt2020}, we developed a framework for multisymplectic
discretization of time-independent Hamiltonian PDEs by hybrid finite
element methods, including hybridizable discontinuous Galerkin methods
(cf.~\citet*{CoGoLa2009}). In this section, we show that those same
methods may be used for semidiscretization of time-dependent
Hamiltonian PDEs, and that a semidiscrete multisymplectic conservation
law holds. Consequently, when combined with a symplectic numerical
integrator for time discretization, the resulting method satisfies a
fully discrete multisymplectic conservation law in spacetime. Unlike
the methods discussed in the previous section, these methods may be
applied to unstructured meshes on non-rectangular domains.

Suppose that $\Omega \subset \mathbb{R}^m $ is polyhedral, and let
$ \mathcal{T} _h $ be a simplicial triangulation of $\Omega$ by
$m$-simplices $ K \in \mathcal{T} _h $, where
$ \mathcal{E} _h = \bigcup _{ K \in \mathcal{T} _h } \partial K $
denotes the set of $ ( m -1 ) $-dimensional facets. We specify finite
element spaces
\begin{alignat*}{2}
V (K) &\subset \bigl[ H ^2 (K) \bigr] ^n , \qquad &V &\coloneqq \prod _{ K \in \mathcal{T} _h } V (K) ,\\
\Sigma (K) &\subset \bigl[ H ^1 (K) \bigr] ^{ m n } , \qquad &\Sigma &\coloneqq \prod _{ K \in \mathcal{T} _h } \Sigma (K) ,
\intertext{along with spaces of approximate boundary traces on
$ \mathcal{E} _h $,}
  \widehat{ V } &\subset \bigl[ L ^2 ( \mathcal{E} _h ) \bigr] ^n , \qquad & \widehat{ V } _0 &\coloneqq \bigl\{ \widehat{ v } \in \widehat{ V } : \widehat{ v } \rvert _{ \partial \Omega } = 0 \bigr\} .
\end{alignat*}
The de~Donder--Weyl equations \eqref{eqn:ddw_time} are then
approximated by the weak problem: Find
$ \bigl( u (t) , \sigma (t) , p (t) , \widehat{ u } (t) \bigr) \in V
\times \Sigma \times V ^\ast \times \widehat{ V } $ satisfying
\begin{subequations}
  \label{eqn:ddw_hdg}
  \begin{alignat}{2}
    \int _K \dot{ u } ^i r _i \,\mathrm{d}^m x &= \int _K \frac{ \partial H }{ \partial p _i } r _i \,\mathrm{d}^m x , \quad &\forall r &\in V ^\ast (K) , \label{eqn:ddw_hdg_r}\\
    0 &= \int _K  \biggl( u ^i \partial _\mu \tau _i ^\mu + \frac{ \partial H }{ \partial \sigma _i ^\mu } \tau _i ^\mu \biggr) \,\mathrm{d}^m x - \int _{ \partial K } \widehat{ u } ^i \tau _i ^\mu \,\mathrm{d}^{m-1} x_\mu , \quad &\forall \tau &\in \Sigma (K) , \label{eqn:ddw_hdg_tau}\\
    \int _K \dot{ p } _i v ^i  \,\mathrm{d}^m x &= \int _K \biggl( \sigma _i ^\mu \partial _\mu v ^i - \frac{ \partial H }{ \partial u ^i } v ^i \biggr) \,\mathrm{d}^m x - \int _{ \partial K } \widehat{ \sigma } _i ^\mu v ^i \,\mathrm{d}^{m-1} x_\mu , \quad & \forall v & \in V (K) \label{eqn:ddw_hdg_v},
    \intertext{for all $ K \in \mathcal{T} _h $, together with the \emph{conservativity condition}}
    0 &= \sum _{ K \in \mathcal{T} _h } \int _{ \partial K } \widehat{ \sigma } _i ^\mu \widehat{ v } ^i \,\mathrm{d}^{m-1} x_\mu , \quad & \forall \widehat{ v } &\in \widehat{ V } _0 \label{eqn:ddw_hdg_vhat}.
  \end{alignat}
\end{subequations}
Here, $ \widehat{ \sigma } $ is determined by
$ u , \sigma, \widehat{ u } $ through a specified numerical flux
function; see \citet*{CoGoLa2009,McSt2020} for further details. The
equations \eqref{eqn:ddw_hdg_r}--\eqref{eqn:ddw_hdg_v} are derived by
multiplying \eqref{eqn:ddw_time} by test functions, integrating by
parts over $K$, and replacing the boundary traces of $u$ and $\sigma$
by the approximate traces $ \widehat{ u } $ and
$ \widehat{ \sigma } $. Under appropriate nondegeneracy assumptions,
the equations \eqref{eqn:ddw_hdg_r} and \eqref{eqn:ddw_hdg_v} define
the dynamics of $ y _h = ( u , p ) $ on
$ Y _h \coloneqq V \times V ^\ast $, where $ \sigma $,
$ \widehat{ u } $, and $ \widehat{ \sigma } $ are implicit functions
of $ y _h $.

We may then consider variations of solutions to \eqref{eqn:ddw_hdg},
along with a corresponding semidiscrete multisymplectic conservation
law in the integral form \eqref{eqn:mscl_int}. The following is a
straightforward generalization of Lemma~2 in \citep{McSt2020}.

\begin{theorem}
  \label{thm:multisymplectic}
  If \eqref{eqn:ddw_hdg_r}--\eqref{eqn:ddw_hdg_v} hold on
  $ K \in \mathcal{T} _h $, then
  \begin{equation*}
    \int _K \partial _t ( \mathrm{d} u ^i \wedge \mathrm{d} p _i ) \,\mathrm{d}^m x = - \int _{ \partial K } ( \mathrm{d} \widehat{ u } ^i \wedge \mathrm{d} \widehat{ \sigma } _i ^\mu ) \,\mathrm{d}^{m-1} x_\mu + \int _{ \partial K } \bigl( \mathrm{d} ( \widehat{ u } ^i - u ^i ) \wedge \mathrm{d} ( \widehat{ \sigma } _i ^\mu - \sigma _i ^\mu ) \bigr) \,\mathrm{d}^{m-1} x_\mu .
  \end{equation*}
  Consequently, the semidiscrete multisymplectic conservation law
  \begin{equation}
    \label{eqn:mscl_hdg}
    \int _K \partial _t ( \mathrm{d} u ^i \wedge \mathrm{d} p _i ) \,\mathrm{d}^m x = - \int _{ \partial K } ( \mathrm{d} \widehat{ u } ^i \wedge \mathrm{d} \widehat{ \sigma } _i ^\mu ) \,\mathrm{d}^{m-1} x_\mu
  \end{equation}
  holds on $ K \in \mathcal{T} _h $ if and only if
  $ \int _{ \partial K } \bigl( \mathrm{d} ( \widehat{ u } ^i - u ^i )
  \wedge \mathrm{d} ( \widehat{ \sigma } _i ^\mu - \sigma _i ^\mu )
  \bigr) \,\mathrm{d}^{m-1} x_\mu = 0 $.
\end{theorem}

\begin{proof}
  We may rewrite \eqref{eqn:ddw_hdg_r}--\eqref{eqn:ddw_hdg_v} as
  \begin{align*}
    \int _K \dot{ u } ^i \,\mathrm{d} p _i \,\mathrm{d}^m x &= \int _K \frac{ \partial H }{ \partial p _i } \,\mathrm{d} p _i \,\mathrm{d}^m x ,\\
    0 &= \int _K  \biggl( u ^i \,\mathrm{d} ( \partial _\mu \sigma _i ^\mu) + \frac{ \partial H }{ \partial \sigma _i ^\mu } \,\mathrm{d} \sigma _i ^\mu \biggr) \,\mathrm{d}^m x - \int _{ \partial K } \widehat{ u } ^i \,\mathrm{d} \sigma _i ^\mu \,\mathrm{d}^{m-1} x_\mu , \\
    \int _K \dot{ p } _i \,\mathrm{d} u ^i  \,\mathrm{d}^m x &= \int _K \biggl( \sigma _i ^\mu \mathrm{d} (\partial _\mu u ^i) - \frac{ \partial H }{ \partial u ^i } \,\mathrm{d} u ^i \biggr) \,\mathrm{d}^m x - \int _{ \partial K } \widehat{ \sigma } _i ^\mu \,\mathrm{d} u ^i \,\mathrm{d}^{m-1} x_\mu .
  \end{align*}
  Adding the first two equations, subtracting the third, and taking
  the exterior derivative on both sides, we get
  \begin{multline*}
    \int _K \partial _t ( \mathrm{d} u ^i \wedge \mathrm{d} p _i ) \,\mathrm{d}^m x \\
    \begin{aligned}
      &= \int _K \bigl( \partial _\mu ( \mathrm{d} u ^i \wedge \mathrm{d} \sigma _i ^\mu ) + \mathrm{d} \mathrm{d} H \bigr) \,\mathrm{d}^m x - \int _{ \partial K } ( \mathrm{d} \widehat{ u } ^i \wedge \mathrm{d} \sigma _i ^\mu + \mathrm{d} u ^i \wedge \mathrm{d} \widehat{ \sigma } _i ^\mu ) \,\mathrm{d}^{m-1} x_\mu \\
      &= \int _{ \partial K } ( \mathrm{d} u ^i \wedge \mathrm{d} \sigma _i ^\mu - \mathrm{d} \widehat{ u } ^i \wedge \mathrm{d} \sigma _i ^\mu - \mathrm{d} u ^i \wedge \mathrm{d} \widehat{ \sigma } _i ^\mu ) \,\mathrm{d}^{m-1} x_\mu \\
      &= - \int _{ \partial K } ( \mathrm{d} \widehat{ u } ^i \wedge \mathrm{d} \widehat{ \sigma } _i ^\mu ) \,\mathrm{d}^{m-1} x_\mu + \int _{ \partial K } \bigl( \mathrm{d} ( \widehat{ u } ^i - u ^i ) \wedge \mathrm{d} ( \widehat{ \sigma } _i ^\mu - \sigma _i ^\mu ) \bigr) \,\mathrm{d}^{m-1} x_\mu ,
    \end{aligned}
  \end{multline*}
  where the second equality uses $ \mathrm{d} \mathrm{d} H = 0 $ and
  the divergence theorem.
\end{proof}

In Section 4 of \citep{McSt2020}, it is proved that several families
of hybrid finite element methods, including hybridized mixed methods
(RT-H and BDM-H), nonconforming methods (NC-H), discontinuous Galerkin
methods (LDG-H and IP-H), and continuous Galerkin methods (CG-H)
satisfy the condition
$ \int _{ \partial K } \bigl( \mathrm{d} ( \widehat{ u } ^i - u ^i )
\wedge \mathrm{d} ( \widehat{ \sigma } _i ^\mu - \sigma _i ^\mu )
\bigr) \,\mathrm{d}^{m-1} x_\mu = 0 $ of
\cref{thm:multisymplectic}. Therefore, when these methods are applied
to \eqref{eqn:ddw_hdg}, they satisfy the semidiscrete multisymplectic
conservation law \eqref{eqn:mscl_hdg} on each
$ K \in \mathcal{T} _h $.

If the numerical flux satisfies the so-called \emph{strong
  conservativity condition}
$ \llbracket \widehat{ \sigma } \rrbracket = 0 $, which is stronger
than \eqref{eqn:ddw_hdg_vhat}, then the multisymplectic conservation
law \eqref{eqn:mscl_hdg} may also be strengthened so that it holds for
arbitrary unions of simplices. This holds for all of the methods
mentioned in the previous paragraph \emph{except} CG-H. The following
is a straightforward generalization of Theorem~3 in \citep{McSt2020}.

\begin{theorem}
  If a strongly conservative method satisfies \eqref{eqn:mscl_hdg},
  then for all $ \mathcal{K} \subset \mathcal{T} _h $,
  \begin{equation*}
    \int _{ \bigcup \mathcal{K} } \partial _t ( \mathrm{d} u ^i \wedge \mathrm{d} p _i ) \,\mathrm{d}^m x = - \int _{ \partial ( \overline{ \bigcup K }) } ( \mathrm{d} \widehat{ u } ^i \wedge \mathrm{d} \widehat{ \sigma } _i ^\mu ) \,\mathrm{d}^{m-1} x_\mu .
  \end{equation*}
\end{theorem}

\begin{proof}
  Sum \eqref{eqn:mscl_hdg} over $ K \in \mathcal{K} $, using
  $ \llbracket \widehat{ \sigma } \rrbracket = 0 $ to cancel the
  contributions of internal facets.
\end{proof}

\begin{remark}
  In the situation considered in \cref{rmk:mscl_symplectic}, taking
  $ \mathcal{K} = \mathcal{T} _h $ implies conservation of the
  symplectic form
  $ \int _\Omega \mathrm{d} u ^i \wedge \mathrm{d} p _i \,\mathrm{d}^m
  x $ on $ Y _h $. This generalizes a result of
  \citet{SaCiNgPeCo2017}, which states that semidiscretization of the
  acoustic wave equation by LDG-H is symplectic.
\end{remark}

\section{Generalization to additive and partitioned methods}
\label{sec:additive_partitioned}

In the preceding sections, we have developed a theory of functional
equivariance for a class of numerical integrators, including B-series
methods, and applied it to local conservation laws for PDEs. This
section extends the functional equivariance theory from \cref{sec:fe}
to two larger classes of numerical integrators: additive methods and
partitioned methods. It follows that, when these methods are applied
to PDEs satisfying local conservation laws, the results of
\cref{sec:conservation} and \cref{sec:multisymplectic} may also be
extended to these classes of methods.

\subsection{Additive methods}

We now consider integrators applied to a vector field
$ f \in \mathfrak{X} (Y) $ after it has been additively decomposed as
$ f = f ^{ [1] } + \cdots + f ^{ [N] } $.  Specifically, we have in
mind additive Runge--Kutta and NB-series methods
(cf.~\citet*{ArMuSa1997}), as well as splitting and composition
methods (cf.~\citet{McQu2002}).

Denote the application of a method $\Phi$ to a decomposed vector field
$ f = f ^{ [1] } + \cdots + f ^{ [N] } $ by
$ \Phi _{ f ^{ [1] } , \ldots , f ^{ [N] } } $. By an additive
numerical integrator, we mean the entire collection of maps
$ \Phi = \bigl\{ \Phi _{ f ^{ [1] } , \ldots , f ^{ [N] } } : f ^{ [1]
} , \ldots , f ^{ [N] } \in \mathfrak{X} (Y) ,\ Y \text{ a Banach
  space} \bigr\} $. We begin by extending the definitions of affine
equivariance and functional equivariance to such methods.

\begin{definition}
  An additive numerical integrator $\Phi$ is \emph{$N$-affine
    equivariant} if
  $ A \circ \Phi _{ f ^{ [1] } , \ldots, f ^{ [N] } } = \Phi _{ g ^{
      [1] } , \ldots, g ^{ [N ]} } \circ A $ whenever
  $ f ^{ [\nu] } \in \mathfrak{X} (Y) $ and
  $ g ^{ [\nu] } \in \mathfrak{X} (U) $ are $A$-related for all
  $ \nu = 1 , \ldots, N $, all affine maps
  $A \colon Y \rightarrow U $, and all Banach spaces $Y$ and $U$.
\end{definition}

\begin{definition}
  Given a G\^ateaux differentiable map $ F \colon Y \rightarrow Z $
  and $ f ^{ [1] } , \ldots, f ^{ [N] } \in \mathfrak{X} (Y) $, define
  $ g ^{ [1] }, \ldots, g ^{ [N] } \in \mathfrak{X} ( Y \times Z ) $
  by
  $ g ^{ [\nu] } ( y, z ) = \bigl( f ^{[\nu]} (y) , F ^\prime (y) f
  ^{[\nu]} (y) \bigr) $ for $ \nu = 1 , \ldots, N $. We say that an
  additive numerical integrator $\Phi$ is \emph{$F$-functionally
    equivariant} if
  $ (\mathrm{id}, F ) \circ \Phi _{ f ^{ [1] } , \ldots, f ^{ [N] } }
  = \Phi _{ g ^{ [1] } , \ldots, g ^{ [N] } } \circ ( \mathrm{id}, F )
  $ for all $ f ^{[1]} , \ldots, f^{[N]} \in \mathfrak{X} (Y) $ and
  \emph{$ \mathcal{F} $-functionally equivariant} if this holds for
  all $ F \in \mathcal{F} ( Y, Z ) $ and all Banach spaces $Y$ and $Z$.
\end{definition}

\begin{proposition}
  \label{prop:N-afe}
  Every $N$-affine equivariant method is affine functionally equivariant.
\end{proposition}

\begin{proof}
  The proof is essentially identical to that for \cref{prop:afe}. If
  $F$ is affine, then so is $ ( \mathrm{id}, F ) $, and the vector
  fields $ f ^{ [\nu] } $ and $ g ^{ [\nu] } $ are
  $ ( \mathrm{id}, F ) $-related for all $ \nu = 1 , \ldots, N $.
\end{proof}

\begin{example}[additive Runge--Kutta methods]
  \label{ex:ark}
  An $s$-stage additive Runge--Kutta (ARK) method has the form
  \begin{align*}
    Y _i &= y _0 + \Delta t \sum _{ \nu = 1 } ^N \sum _{ j = 1 } ^s a _{ i j } ^{[\nu] }f ^{[\nu]} ( Y _j ) , \qquad  i = 1, \ldots, s ,\\
    y _1 &= y _0 + \Delta t \sum _{ \nu = 1 } ^N \sum _{ i = 1 } ^s b _i ^{[\nu]} f ^{[\nu]} ( Y _i ) ,
  \end{align*}
  and $F$-functional equivariance is the condition
  \begin{equation*}
    F ( y _1 ) = F ( y _0 ) + \Delta t \sum _{ \nu = 1 } ^N \sum _{ i = 1 } ^s b _i ^{[\nu]} F ^\prime ( Y _i ) f ^{[\nu]} ( Y _i ) .
  \end{equation*}
  If $F$ is an invariant, then we have
  $ F ^\prime ( Y _i ) f ( Y _i ) = 0 $ but generally
  $ F ^\prime ( Y _i ) f ^{ {[\nu]} } ( Y _i ) \neq 0 $ for $ N > 1 $,
  so the sum on the right-hand side need not vanish.  However, if
  $ b _i ^{[\nu]} = b _i $ is independent of $\nu$, then it does
  vanish, and we obtain $ F ( y _1 ) = F ( y _0 ) $ as in
  \cref{ex:rk}. This illustrates that an ARK method may be
  functionally equivariant but not invariant preserving (even for
  affine maps) unless some additional condition is satisfied.
\end{example}

\begin{proposition}
  \label{prop:ark}
  Additive Runge--Kutta methods are $N$-affine
  equivariant. Furthermore, an ARK method preserves affine invariants
  if $ b _i ^{[\nu]} = b _i $ is independent of $\nu$.
\end{proposition}

\begin{proof}
  Suppose $ f ^{ [\nu] } $ and $ g ^{ [\nu] } $ are $A$-related for
  $ \nu = 1 , \ldots, N $. Then
  \begin{align*}
    A ( Y _i ) 
    &= A ( y _0 ) + A ^\prime ( Y _i - y _0 ) \\
    &= A ( y _0 ) + \Delta t \sum _{ \nu = 1 } ^N \sum _{ j = 1 } ^s a _{ i j } ^{[\nu] } ( A ^\prime \circ f ^{[\nu]} ) ( Y _j ) \\
    &= A ( y _0 ) + \Delta t \sum _{ \nu = 1 } ^N \sum _{ j = 1 } ^s a _{ i j } ^{[\nu] } g ^{[\nu]} \bigl( A ( Y _j ) \bigr) ,
  \end{align*}
  for $ i = 1 , \ldots, s $, and similarly,
  \begin{equation*}
    A ( y _1 ) = A ( y _0 ) + \Delta t \sum _{ \nu = 1 } ^N \sum _{ i = 1 } ^s b _i ^{[\nu]} g ^{[\nu]} \bigl( A ( Y _i ) \bigr) .
  \end{equation*} 
  This shows that
  $ A ( y _1 ) = ( A \circ \Phi _f ) ( y _0 ) = ( \Phi _g \circ A ) (
  y _0 ) $, so $ \Phi $ is $N$-affine equivariant. Finally, if
  $ b _i ^{[\nu]} = b _i $ is independent of $\nu$, then
  \cref{prop:N-afe} and \cref{ex:ark} show that $\Phi$ preserves
  affine invariants.
\end{proof}

\begin{remark}
  It is straightforward to show that, in fact, all NB-series methods
  are $N$-affine equivariant. (This includes, e.g., generalized
  additive Runge--Kutta methods, whose symplecticity conditions were
  recently investigated by \citet*{GuSaZa2021}.) The proof is,
  essentially, to repeatedly differentiate the $A$-relatedness
  condition $ A ^\prime \circ f ^{ [\nu] } = g ^{ [\nu] } \circ A $,
  obtaining a relation between the elementary
  differentials.
\end{remark}

\begin{theorem}
  \label{thm:n-fe}
  Let $\mathcal{F}$ satisfy \cref{assumption:affine}. An additive
  numerical integrator $\Phi$ preserves $\mathcal{F}$-invariants if
  and only if it is $\mathcal{F}$-functionally equivariant and
  preserves affine invariants.
\end{theorem}

\begin{proof}
  $ ( \Rightarrow ) $ Suppose $\Phi$ preserves
  $\mathcal{F}$-invariants. The proof of $\mathcal{F}$-functional
  equivariance is essentially identical to that in \cref{thm:fe}, and
  preservation of affine invariants follows from the fact that
  $\mathcal{F}$ contains affine maps by \cref{assumption:affine}.

  $ ( \Leftarrow ) $ Conversely, suppose that $\Phi$ is
  $\mathcal{F}$-functionally equivariant and preserves affine
  invariants. If $ F \in \mathcal{F} ( Y , Z ) $ is an invariant of
  $ f \in \mathfrak{X} ( Y ) $, then
  $ g ^{ [\nu] } (y, z ) = \bigl( f ^{[\nu]} (y) , F ^\prime (y) f ^{
    [\nu] } (y) \bigr) $ is the corresponding decomposition of
  $ g = ( f, 0 ) $. By $\mathcal{F}$-functional equivariance, we have
  $ \Phi _{ g ^{ [1] } , \ldots, g ^{ [N] } } \colon \bigl( y _0, F (
  y _0 ) \bigr) \mapsto \bigl( y _1, F ( y _1 ) \bigr) $. Finally,
  since $ G ( y, z ) = z $ is an affine invariant of $g$, it is
  preserved by $ \Phi _{ g ^{ [1] } , \ldots, g ^{ [N] } } $, and thus
  $ F ( y _0 ) = F ( y _1 ) $.
\end{proof}

\begin{example}
  \label{ex:ark_quadratic}
  Let $\mathcal{F}$ be the class of quadratic maps. It follows that an
  additive numerical integrator preserves quadratic invariants if and
  only if it is quadratic functionally equivariant and preserves
  affine invariants. For ARK methods, a sufficient condition is that
  $ b _i ^{ [\nu ] } = b _i $ be independent of $\nu$ and
  $ b _i ^{ [\nu] } a _{ i j } ^{[\mu]} + b _j ^{[\mu]} a _{ j i }
  ^{[\nu]} = b _i ^{ [\nu] } b _j ^{ [\mu] } $ for all $i$, $j$,
  $\mu$, $\nu$. The proof is identical to that for symplecticity of
  ARK methods, cf.~\citet*[Theorem 7]{ArMuSa1997}.
\end{example}

Splitting methods take $ \Phi _{ f ^{ [1] }, \cdots , f ^{ [N] } } $
to be a composition of exact flows $ \varphi _{ \tau f ^{ [\nu] } } $,
i.e.,
\begin{equation*}
  \Phi _{ f ^{ [1] } , \ldots, f ^{ [N] } } = \varphi _{ \tau _s f ^{ [\nu _s] } } \circ \cdots \circ \varphi _{ \tau _1 f ^{ [\nu _1] } } ,
\end{equation*}
where consistency requires $ \sum _{ \nu _i = \nu } \tau _i = 1 $ for
all $ \nu = 1 , \ldots, N $. For $ N = 2 $, the two most elementary
splitting methods are the Lie--Trotter splitting
$ \varphi _{ f ^{ [1] } } \circ \varphi _{ f ^{ [2] } } $ and the
Strang splitting
$ \varphi _{ \frac{1}{2} f ^{ [2] } } \circ \varphi _{ f ^{ [1] } }
\circ \varphi _{ \frac{1}{2} f ^{ [2] } } $, where $\varphi$ denotes
the exact time-$1$ flow. Since the exact flow is equivariant (and
hence functionally equivariant) with respect to \emph{all} maps $F$,
the chain rule implies that this is also true of splitting methods. As
a consequence of \cref{thm:n-fe}, we get the following
negative result for splitting methods.

\begin{corollary}
  Any splitting method that preserves affine invariants equals the
  exact flow.
\end{corollary}

\begin{proof}
  Since splitting methods are equivariant with respect to all maps,
  \cref{thm:n-fe} implies that any splitting method
  preserving affine invariants preserves \emph{all} invariants. To see
  that this must be the exact flow, consider the vector field
  $ ( f, 1 ) \in \mathfrak{X} ( Y \times \mathbb{R} ) $, which
  augments $ \dot{y} = f (y) $ by the equation $ \dot{ t } = 1 $. The
  exact solution is $ y (t) = \varphi _{ t f } (y _0) $, so
  $ F ( y, t ) = y - \varphi _{ t f } (y _0) $ is an invariant of
  $ ( f, 1 ) $. Therefore, $ F ( y _1 , 1 ) = F ( y _0 , 0 ) = 0 $,
  which says that $ y _1 = \varphi _f ( y _0 ) $.
\end{proof}

\subsection{Partitioned methods}
\label{sec:partitioned}

We finally consider partitioned methods, which are based on a
partitioning $ Y = Y ^{ [1] } \oplus \cdots \oplus Y ^{ [N] } $. In
particular, we have in mind partitioned Runge--Kutta and P-series
methods (cf.~\citet{Hairer1980}). These are closely related to the
methods in the previous section, except the vector field decomposition
$ f = f ^{ [1] } + \cdots + f ^{ [N] } $ is uniquely specified by the
partitioning of $Y$, i.e., $ f ^{ [\nu] } (y) \in Y ^{ [\nu] } $ for
all $ y \in Y $ and $ \nu = 1 , \ldots, N $. For this reason, we write
the flow of such a method as $ \Phi _f $ rather than
$ \Phi _{ f ^{ [1] } , \ldots, f ^{ [N] } } $. By a partitioned
numerical integrator, we mean the entire collection of maps
$ \Phi = \bigl\{ \Phi _f : f \in \mathfrak{X} ( Y ),\ Y = \bigoplus _{
  \nu = 1 } ^N Y ^{[\nu]} \text{ a partitioned Banach space} \bigr\}
$.

\begin{definition}
  Given partitioned spaces
  $ Y = \bigoplus _{ \nu = 1 } ^N Y ^{[\nu]} $ and
  $ U = \bigoplus _{ \nu = 1 } ^N U ^{ [\nu] } $, we say that
  $ A \colon Y \rightarrow U $ is a \emph{P-affine map} if it
  decomposes as $ A = \bigoplus _{ \nu = 1 } ^N A ^{ [\nu] } $, where
  each $ A ^{ [\nu] } \colon Y ^{ [\nu] } \rightarrow U ^{ [\nu] } $
  is affine. A partitioned numerical integrator $\Phi$ is \emph{P-affine
    equivariant} if $ A \circ \Phi _f = \Phi _g \circ A $ whenever
  $ f ^{ [\nu] } $ and $ g ^{ [\nu] } $ are $ A $-related
  for all $ \nu = 1 , \ldots, N $, all P-affine maps $A$, all
  partitionings, and all Banach spaces $Y$ and $U$.
\end{definition}

\begin{example}
  \label{ex:p-affine}
  If we partition $ U = \mathbb{R} $ into
  $ U ^{ [\mu] } = \mathbb{R} $ and $ U ^{ [\nu] } = \{ 0 \} $ for
  $ \nu \neq \mu $, then the P-affine functionals are those
  depending only on $ Y ^{ [\mu] } $. Affine functionals depending on
  more than one component $ Y ^{ [\nu] } $ cannot be P-affine for
  any partitioning of $\mathbb{R}$.  In particular, if we take
  $ Y = \mathbb{R}^2 = \bigl( \mathbb{R} \times \{ 0 \} \bigr) \oplus
  \bigl( \{ 0 \} \times \mathbb{R} \bigr) $, then:
  \begin{itemize}
  \item $ ( q, p ) \mapsto q $ is P-affine for the partitioning $ U = \mathbb{R}  \oplus \{ 0 \} $;
  \item $ ( q, p ) \mapsto p $ is P-affine for the partitioning
    $ U = \{0 \} \oplus \mathbb{R} $;
  \item $ ( q, p ) \mapsto q + p $ is never P-affine.
  \end{itemize} 
\end{example}

\begin{proposition}
  \label{prop:n-affine_implies_p-affine}
  If an additive numerical integrator $ \Psi $ is $N$-affine
  equivariant, then the partitioned numerical integrator $\Phi$
  defined by $\Phi _f = \Psi _{ f ^{ [1] } , \ldots, f ^{ [N] } } $ is
  P-affine equivariant.
\end{proposition}

\begin{proof}
  This follows immediately from the definitions, since P-affine maps
  are affine.
\end{proof}

\begin{example}[partitioned Runge--Kutta methods]
  An $s$-stage partitioned Runge--Kutta method (PRK) is just the
  application of an ARK method to a partitioned space, as in
  \cref{prop:n-affine_implies_p-affine}, where $\Phi$ is the PRK
  method and $ \Psi $ is the ARK method. As an immediate corollary of
  this proposition, all PRK methods are P-affine equivariant.
\end{example}

The definition of $F$- and $\mathcal{F}$-functional equivariance is
the same as in \cref{def:fe}, where given
$ Y = \bigoplus _{ \nu = 1 } ^N Y ^{[\nu]} $ and
$ Z = \bigoplus _{ \nu = 1 } ^N Z ^{[\nu]} $, we partition
$ Y \times Z = \bigoplus _{ \nu = 1 } ^N ( Y ^{[\nu]} \times Z ^{
  [\nu] } ) $. However, the methods being considered are not
necessarily equivariant with respect to all affine maps, so
\cref{assumption:affine} is too restrictive on $\mathcal{F}$. We
therefore replace it with the following, which just replaces
``affine'' by ``P-affine'' for specified partitions.

\begin{assumption}
  \label{assumption:p-affine}
  Assume that:
  \begin{itemize}
  \item $\mathcal{F} (Y,Y) $ contains the identity map for all
    $Y = \bigoplus _{ \nu = 1 } ^N Y ^{[\nu]}$;
  \item $ \mathcal{F} ( Y, Z ) $ is a vector space for all
    $Y = \bigoplus _{ \nu = 1 } ^N Y ^{[\nu]}$ and
    $Z=\bigoplus _{ \nu = 1 } ^N Z ^{[\nu]}$;
  \item $ \mathcal{F} $ is invariant under composition with P-affine
    maps, in the following sense: If $ A \colon Y \rightarrow U $ and
    $ B \colon V \rightarrow Z $ are P-affine and
    $ F \in \mathcal{F} ( U, V ) $, then
    $ B \circ F \circ A \in \mathcal{F} ( Y, Z ) $, for all
    $Y = \bigoplus _{ \nu = 1 } ^N Y ^{[\nu]}$,
    $Z = \bigoplus _{ \nu = 1 } ^N Z ^{[\nu]}$,
    $U = \bigoplus _{ \nu = 1 } ^N U ^{[\nu]}$, and
    $V = \bigoplus _{ \nu = 1 } ^N V ^{[\nu]}$.
  \end{itemize}
\end{assumption}

\begin{theorem}
  \label{thm:p-fe}
  Let $\mathcal{F}$ satisfy \cref{assumption:p-affine}. A partitioned
  numerical integrator $\Phi$ preserves $\mathcal{F}$-invariants if
  and only if it is $\mathcal{F}$-functionally equivariant.
\end{theorem}

\begin{proof}
  The proof is formally identical to that for \cref{thm:fe}.
\end{proof}

\begin{example}
  Let $\mathcal{F}$ be the class of P-affine maps. It follows that all
  P-affine equivariant methods preserve P-affine invariants. In
  particular, by \cref{ex:p-affine}, affine invariants
  $ F \colon Y \rightarrow \mathbb{R} $ depending only on a single
  component $ Y ^{ [\mu] } $ are preserved.
\end{example}

\begin{example}
  \label{ex:prk_affine}
  Let $\mathcal{F}$ be the class of \emph{all} affine maps,
  irrespective of partitioning. It follows that P-affine equivariant
  methods preserve affine invariants if and only if they are affine
  functionally equivariant. For PRK methods, as for ARK methods, this
  holds if $ b _i ^{ [\nu] } = b _i $ is independent of $\nu$. (See
  \cref{ex:ark,prop:ark}.)
\end{example}

\begin{example}
  \label{ex:bilinear}
  Let $\mathcal{F}$ be the class of quadratic maps that are at most
  bilinear with respect to the partition. i.e., terms may be bilinear
  in $ y ^{ [\mu] } $ and $ y ^{ [\nu] } $ for $ \mu \neq \nu $. For
  PRK methods, a sufficient condition for $\mathcal{F}$-invariant
  preservation, and thus for $\mathcal{F}$-functional equivariance, is
  that $ b _i ^{ [\nu ] } = b _i $ be independent of $\nu$ and
  $ b _i ^{ [\nu] } a _{ i j } ^{[\mu]} + b _j ^{[\mu]} a _{ j i }
  ^{[\nu]} = b _i ^{ [\nu] } b _j ^{ [\mu] } $ for all $i$, $j$, and
  $ \mu \neq \nu $. This is a straightforward generalization of the
  $ N = 2 $ case, cf.~\citet*[Theorem IV.2.4]{HaLuWa2006}.
\end{example}

\begin{example}
  Let $\mathcal{F}$ be the class of \emph{all} quadratic maps,
  irrespective of partitioning. For PRK methods, as for ARK methods, a
  sufficient condition for quadratic invariant preservation, and thus
  for quadratic functional equivariance, is that
  $ b _i ^{ [\nu ] } = b _i $ be independent of $\nu$ and
  $ b _i ^{ [\nu] } a _{ i j } ^{[\mu]} + b _j ^{[\mu]} a _{ j i }
  ^{[\nu]} = b _i ^{ [\nu] } b _j ^{ [\mu] } $ for all $i$, $j$,
  $\mu$, $\nu$. (See \cref{ex:ark_quadratic}.)
\end{example}

\subsection{Closure under differentiation and (multi)symplecticity}

Finally, we generalize \cref{thm:differentiation}, which allows the
functional equivariance results for $N$-affine and P-affine
equivariant methods to be applied to observables depending on
variations.

\begin{theorem}
  \label{thm:differentiation_additive_partitioned}
  $N$-affine and P-affine equivariant methods are closed under differentiation.
\end{theorem}

\begin{proof}
  The proof is basically the same as \cref{thm:differentiation},
  although we need to specify how $\Phi$ is applied to the augmented
  system
  \begin{equation*}
    \dot{x} = f (x) , \qquad \dot{y} = f (y) , \qquad \dot{z} = \frac{ f (x) - f (y) }{ \epsilon } .
  \end{equation*}
  We simply use the same decomposition or partition for each of the
  three parts. Specifically, if $\Phi$ is $N$-affine equivariant, then
  we decompose
  \begin{equation*}
    f (x) = \sum _{ \nu = 1 } ^N f ^{ [\nu] } (x) ,\qquad f (y) = \sum _{ \nu = 1 } ^N f ^{ [\nu] } (y) , \qquad \frac{ f (x) - f (y) }{ \epsilon } = \sum _{ \nu = 1 } ^N \frac{ f ^{ [\nu] } (x) - f ^{ [\nu] } (y) }{ \epsilon } ,
  \end{equation*}
  while if $\Phi$ is P-affine equivariant, we partition
  $ Y \times Y \times Y = \bigoplus _{ \nu = 1 } ^N ( Y ^{ [\nu] }
  \times Y ^{ [\nu] } \times Y ^{ [\nu] } ) $. The proof then proceeds
  as in \cref{thm:differentiation}.
\end{proof}

Therefore, the results on symplecticity and multisymplecticity of
affine equivariant methods preserving quadratic invariants also hold
for $N$-affine and P-affine equivariant methods preserving quadratic
invariants. Moreover, since the canonical symplectic form
$ \omega = \mathrm{d} q ^i \wedge \mathrm{d} p _i $ and
multisymplectic form
$ \omega ^0 = \mathrm{d} u ^i \wedge \mathrm{d} p _i $ are bilinear on
$ Y = V \times V ^\ast $, it suffices for an $ N = 2 $ partitioned
method to preserve only bilinear invariants, as in
\cref{ex:bilinear}. This includes widely used symplectic PRK methods
such as St\"ormer/Verlet and the Lobatto IIIA--IIIB pair
(\citet*[Sections IV.2 and VI.4]{HaLuWa2006}), as well as compositions
of these methods.

\section{Concluding remarks}

We conclude by posing a natural question for future investigation:
\emph{Which numerical integrators are affine functionally
  equivariant?} Here is a summary of some related results that have
been mentioned throughout this paper:
\begin{itemize}
\item B-series methods are precisely the \emph{affine equivariant}
  methods \citep{McMoMuVe2016}, so by \cref{prop:afe}, they are
  included among the affine functionally equivariant methods.

\item Aromatic B-series methods are precisely the \emph{affine
    isomorphism equivariant} methods \citep{MuVe2016}. Since only
  isomorphsims are considered, the series coefficients may vary
  depending on $ \dim Y $. If the series coefficients \emph{are}
  constant across dimensions, then the method is affine functionally
  equivariant, as in \cref{rmk:aromatic}. Conversely,
  variable-coefficient methods cannot be affine functionally
  equivariant, since $y$ would then evolve differently between the
  original and augmented systems.

\item As shown in \cref{ex:prk_affine}, partitioned methods may also
  be affine functionally equivariant, e.g., a PRK method with
  $ b _i ^{[\nu]} = b _i $ independent of $\nu$. However, such methods
  are generally not affine isomorphism equivariant, e.g., if
  $ a _{ i j } ^{[\nu]} $ varies with $\nu$, so affine functional
  equivariance need not imply affine isomorphism equivariance.
\end{itemize} 
\Cref{fig} depicts these relationships among the different classes of
``equivariant'' methods.

\begin{figure}
  \centering
  \begin{tikzpicture}[scale=2]
    \draw[rounded corners] (1.5,1.5) rectangle (3.5,3.5) {};
    \draw[rounded corners] (0.5,0.5) rectangle (4.5,4.5) {};
    \draw[rounded corners] (1,1) rectangle (6.5,4) {};
    \draw (1,4) node[below right] {Affine isomorphism equivariant methods};
    \draw (1.5,3.5) node[below right,text width=1cm] {Affine\\ equivariant\\ methods};
    \draw (0.5,4.5) node[below right] {Affine functionally equivariant methods};
    \draw (1.5,1.5) node[above right] {\em $\equiv$ B-series};
    \draw (1,1) node[above right, text width=4cm] {\em $\equiv$ constant-coefficient aromatic B-series};
    \draw (0.5,0.5) node[above right] {\em Example: PRK method with $b_i ^{[\nu]} = b_i$};
    \draw (4.5,1) node[above right, text width=4cm] {\em $\equiv$ variable-coefficient aromatic B-series};
  \end{tikzpicture}
  \caption{The landscape of equivariant methods.\label{fig}}
\end{figure}
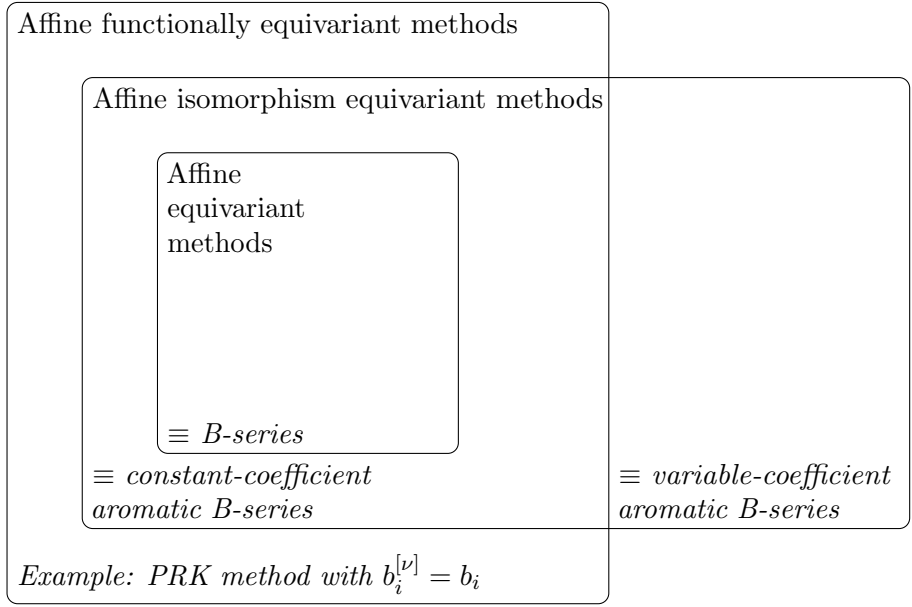
\subsection*{Acknowledgments}

The authors would like to thank the Isaac Newton Institute for
Mathematical Sciences for support and hospitality during the program
``Geometry, compatibility and structure preservation in computational
differential equations,'' when work on this paper was undertaken. This
program was supported by EPSRC grant number EP/R014604/1.

Robert McLachlan was supported in part by the Marsden Fund of the
Royal Society of New Zealand and by a fellowship from the Simons
Foundation. Ari Stern was supported in part by NSF grant DMS-1913272.


\footnotesize
\begin{thebibliography}{30}
\providecommand{\natexlab}[1]{#1}

\bibitem[{Abraham et~al.(1988)Abraham, Marsden, and Ratiu}]{AbMaRa1988}
\textsc{R.~Abraham, J.~E. Marsden, and T.~Ratiu}, \emph{Manifolds, tensor
  analysis, and applications}, vol.~75 of Applied Mathematical Sciences,
  Springer-Verlag, New York, second ed., 1988.

\bibitem[{Ara\'{u}jo et~al.(1997)Ara\'{u}jo, Murua, and
  Sanz-Serna}]{ArMuSa1997}
\textsc{A.~L. Ara\'{u}jo, A.~Murua, and J.~M. Sanz-Serna}, \emph{Symplectic
  methods based on decompositions}, SIAM J. Numer. Anal., 34 (1997), pp.
  1926--1947.

\bibitem[{Ascher and McLachlan(2004)}]{AsMc2004}
\textsc{U.~M. Ascher and R.~I. McLachlan}, \emph{Multisymplectic box schemes
  and the {K}orteweg--de {V}ries equation}, Appl. Numer. Math., 48 (2004), pp.
  255--269.

\bibitem[{Berchenko-Kogan and Stern(2021)}]{BeSt2021}
\textsc{Y.~Berchenko-Kogan and A.~Stern}, \emph{Constraint-preserving hybrid
  finite element methods for {M}axwell's equations}, Found. Comput. Math., 21
  (2021), pp. 1075--1098.

\bibitem[{Bochev and Scovel(1994)}]{BoSc1994}
\textsc{P.~B. Bochev and C.~Scovel}, \emph{On quadratic invariants and
  symplectic structure}, BIT, 34 (1994), pp. 337--345.

\bibitem[{Bridges and Reich(2001)}]{BrRe2001}
\textsc{T.~J. Bridges and S.~Reich}, \emph{Multi-symplectic integrators:
  numerical schemes for {H}amiltonian {PDE}s that conserve symplecticity},
  Phys. Lett. A, 284 (2001), pp. 184--193.

\bibitem[{Burrage and Butcher(1979)}]{BuBu1979}
\textsc{K.~Burrage and J.~C. Butcher}, \emph{Stability criteria for implicit
  {R}unge-{K}utta methods}, SIAM J. Numer. Anal., 16 (1979), pp. 46--57.

\bibitem[{Butcher(1975)}]{Butcher1975}
\textsc{J.~C. Butcher}, \emph{A stability property of implicit {R}unge--{K}utta
  methods}, BIT, 15 (1975), pp. 358--361.

\bibitem[{Chartier and Murua(2007)}]{ChMu2007}
\textsc{P.~Chartier and A.~Murua}, \emph{Preserving first integrals and volume
  forms of additively split systems}, IMA J. Numer. Anal., 27 (2007), pp.
  381--405.

\bibitem[{Cockburn et~al.(2009)Cockburn, Gopalakrishnan, and
  Lazarov}]{CoGoLa2009}
\textsc{B.~Cockburn, J.~Gopalakrishnan, and R.~Lazarov}, \emph{Unified
  hybridization of discontinuous {G}alerkin, mixed, and continuous {G}alerkin
  methods for second order elliptic problems}, SIAM J. Numer. Anal., 47 (2009),
  pp. 1319--1365.

\bibitem[{de~Donder(1935)}]{deDonder1935}
\textsc{T.~de~Donder}, \emph{Th\'eorie Invariantive du Calcul des Variations},
  Gauthier-Villars, second ed., 1935.

\bibitem[{Frasca-Caccia and Hydon(2021)}]{FrHy2021}
\textsc{G.~Frasca-Caccia and P.~E. Hydon}, \emph{A new technique for preserving
  conservation laws}, Foundations of Computational Mathematics,  (2021).
  \url{https://doi.org/10.1007/s10208-021-09511-1}.

\bibitem[{G\"unther et~al.(2021)G\"unther, Sandu, and Zanna}]{GuSaZa2021}
\textsc{M.~G\"unther, A.~Sandu, and A.~Zanna}, \emph{Symplectic {GARK} methods
  for {Hamiltonian} systems}, 2021. Preprint, arXiv:2103.04110 [math.NA].

\bibitem[{Hairer(1980/81)}]{Hairer1980}
\textsc{E.~Hairer}, \emph{Order conditions for numerical methods for
  partitioned ordinary differential equations}, Numer. Math., 36 (1980/81), pp.
  431--445.

\bibitem[{Hairer et~al.(2010)Hairer, Lubich, and Wanner}]{HaLuWa2006}
\textsc{E.~Hairer, C.~Lubich, and G.~Wanner}, \emph{Geometric numerical
  integration}, vol.~31 of Springer Series in Computational Mathematics,
  Springer, Heidelberg, 2010.

\bibitem[{Iserles et~al.(2007)Iserles, Quispel, and Tse}]{IsQuTs2007}
\textsc{A.~Iserles, G.~R.~W. Quispel, and P.~S.~P. Tse}, \emph{B-series methods
  cannot be volume-preserving}, BIT, 47 (2007), pp. 351--378.

\bibitem[{McLachlan and Perlmutter(2001)}]{McPe2001}
\textsc{R.~McLachlan and M.~Perlmutter}, \emph{Conformal {H}amiltonian
  systems}, J. Geom. Phys., 39 (2001), pp. 276--300.

\bibitem[{McLachlan et~al.(2016)McLachlan, Modin, Munthe-Kaas, and
  Verdier}]{McMoMuVe2016}
\textsc{R.~I. McLachlan, K.~Modin, H.~Munthe-Kaas, and O.~Verdier},
  \emph{B-series methods are exactly the affine equivariant methods}, Numer.
  Math., 133 (2016), pp. 599--622.

\bibitem[{McLachlan and Quispel(2001)}]{McQu2001}
\textsc{R.~I. McLachlan and G.~R.~W. Quispel}, \emph{What kinds of dynamics are
  there? {L}ie pseudogroups, dynamical systems and geometric integration},
  Nonlinearity, 14 (2001), pp. 1689--1705.

\bibitem[{McLachlan and Quispel(2002)}]{McQu2002}
\leavevmode\vrule height 2pt depth -1.6pt width 23pt, \emph{Splitting methods},
  Acta Numer., 11 (2002), pp. 341--434.

\bibitem[{McLachlan et~al.(2014)McLachlan, Ryland, and Sun}]{McRySu2014}
\textsc{R.~I. McLachlan, B.~N. Ryland, and Y.~Sun}, \emph{High order
  multisymplectic {R}unge--{K}utta methods}, SIAM Journal on Scientific
  Computing, 36 (2014), pp. A2199--A2226.

\bibitem[{McLachlan and Stern(2020)}]{McSt2020}
\textsc{R.~I. McLachlan and A.~Stern}, \emph{Multisymplecticity of hybridizable
  discontinuous {G}alerkin methods}, Found. Comput. Math., 20 (2020), pp.
  35--69.

\bibitem[{Munthe-Kaas and Verdier(2016)}]{MuVe2016}
\textsc{H.~Munthe-Kaas and O.~Verdier}, \emph{Aromatic {B}utcher series},
  Found. Comput. Math., 16 (2016), pp. 183--215.

\bibitem[{N{\'e}d{\'e}lec(1980)}]{Nedelec1980}
\textsc{J.-C. N{\'e}d{\'e}lec}, \emph{Mixed finite elements in
  {$\mathbb{R}^{3}$}}, Numer. Math., 35 (1980), pp. 315--341.

\bibitem[{Olver(1993)}]{Olver1993}
\textsc{P.~J. Olver}, \emph{Applications of {L}ie groups to differential
  equations}, vol. 107 of Graduate Texts in Mathematics, Springer-Verlag, New
  York, second ed., 1993.

\bibitem[{Reich(2000)}]{Reich2000}
\textsc{S.~Reich}, \emph{Multi-symplectic {R}unge-{K}utta collocation methods
  for {H}amiltonian wave equations}, J. Comput. Phys., 157 (2000), pp.
  473--499.

\bibitem[{Ryland and McLachlan(2008)}]{RyMc2008}
\textsc{B.~N. Ryland and R.~I. McLachlan}, \emph{On multisymplecticity of
  partitioned {R}unge-{K}utta methods}, SIAM J. Sci. Comput., 30 (2008), pp.
  1318--1340.

\bibitem[{S\'anchez et~al.(2017)S\'anchez, Ciuca, Nguyen, Peraire, and
  Cockburn}]{SaCiNgPeCo2017}
\textsc{M.~A. S\'anchez, C.~Ciuca, N.~C. Nguyen, J.~Peraire, and B.~Cockburn},
  \emph{Symplectic {H}amiltonian {HDG} methods for wave propagation phenomena},
  J. Comput. Phys., 350 (2017), pp. 951--973.

\bibitem[{Sun and Xing(2020)}]{SuXi2020}
\textsc{Z.~Sun and Y.~Xing}, \emph{On structure-preserving discontinuous
  {G}alerkin methods for {H}amiltonian partial differential equations: energy
  conservation and multi-symplecticity}, J. Comput. Phys., 419 (2020), pp.
  109662, 25.

\bibitem[{Weyl(1935)}]{Weyl1935}
\textsc{H.~Weyl}, \emph{Geodesic fields in the calculus of variation for
  multiple integrals}, Ann. of Math. (2), 36 (1935), pp. 607--629.

\end{thebibliography}
\end{document}